\documentclass[11pt,reqno]{amsart}

\usepackage{a4wide,amsmath,amssymb,graphicx,enumitem,amsaddr,tikz,pgfplots,hhline,epstopdf,multirow,multicol,soul}
\usepackage[hang,flushmargin]{footmisc}

\usetikzlibrary{decorations.markings}

\newtheorem{theorem}{Theorem}
\newtheorem{lemma}{Lemma}

\setlist[enumerate]{leftmargin=8mm}
\setlist[itemize]{leftmargin=8mm}

\newcommand{\R}{\mathbb{R}}
\newcommand{\N}{\mathbb{N}}
\newcommand{\cP}{\mathcal{P}}
\newcommand{\cQ}{\mathcal{Q}}
\newcommand{\cR}{\mathcal{R}}
\newcommand{\cN}{\mathcal{N}}
\newcommand{\be}{\mathbf{e}}
\newcommand{\bJ}{\mathbf{J}}
\newcommand{\SI}{\Upsilon^{\cR_0}}
\newcommand{\ds}{\frac{\text{d}S}{\text{d}t}}
\newcommand{\di}{\frac{\text{d}I}{\text{d}t}}
\newcommand{\dr}{\frac{\text{d}R}{\text{d}t}}

\DeclareMathOperator{\Inc}{Inc}
\DeclareMathOperator{\Rec}{Rec}

\DeclareMathOperator{\TR}{TR}
\DeclareMathOperator{\HB}{HB}
\DeclareMathOperator{\SN}{SN}
\DeclareMathOperator{\FLC}{FLC}
\DeclareMathOperator{\HM}{HM}
\DeclareMathOperator{\tr}{tr}

\definecolor{darkgreen}{RGB}{34,177,76}

\begin{document}
\title{Mathematical analysis of an epidemic model for COVID-19: how important is the people's cautiousness level for eradication?}
\author{\small Benny Yong, Livia Owen, Jonathan Hoseana}
\address{\normalfont\small Department of Mathematics, Parahyangan Catholic University, Bandung 40141, Indonesia}
\email{benny\_y@unpar.ac.id\textnormal{, }livia.owen@unpar.ac.id\textnormal{, }j.hoseana@unpar.ac.id}
\date{}

\begin{abstract}
We construct an SIR-type model for COVID-19, incorporating as a parameter the susceptible individuals' cautiousness level. We determine the model's basic reproduction number, study the stability of the equilibria analytically, and perform a sensitivity analysis to confirm the significance of the cautiousness level. Fixing specific values for all other parameters, we study numerically the model's dynamics as the cautiousness level varies, revealing backward transcritical, Hopf, and saddle-node bifurcations of equilibria, as well as homoclinic and fold bifurcations of limit cycles with the aid of AUTO. Considering some key events affecting the pandemic in Indonesia, we design a scenario in which the cautiousness level varies over time, and show that the model exhibits a hysteresis, whereby, a slight cautiousness decrease could bring a disease-free state to endemic, and this is reversible only by a drastic cautiousness increase, thereby mathematically justifying the importance of a high cautiousness level for resolving the pandemic.


\end{abstract}

\maketitle

\section{Introduction}\label{section:Introduction}

Since December 2019, the Coronavirus Disease 2019, popularly known as COVID-19, has spread, reportedly from a seafood market in Wuhan, People's Republic of China, to the entire world, with its impact persisting to the present \cite{Saxena}. On March 11th, 2020, the World Health Organization (WHO) declared a pandemic \cite{Saxena}, and as of February 6th, 2022, there were over 392 million reported COVID-19 cases, resulting in over 5.7 million fatalities \cite{WHO}.

A plethora of mathematical models has been used to study the spread of the disease. The Kermack-McKendrick SIR model \cite{KermackMcKendrick}, one of the simplest and best-known compartmental models for epidemics, assumes that the population size is constant, and that, at any given time, each individual belongs to exactly one of the following compartments which indicates the individual's status: susceptible (S), infected (I), and recovered (R), hence the model's name (Figure \ref{fig:SIR}). Denoting by $S=S(t)$, $I=I(t)$, and $R=R(t)$, respectively, the number of susceptible, infected, and recovered individuals at time $t\geqslant 0$, the model reads
$$\left\{\begin{array}{rcl}
\displaystyle\ds&=&\displaystyle-\beta SI,\\[0.25cm]
\displaystyle\di&=&\displaystyle\beta SI-\alpha I,\\[0.25cm]
\displaystyle\dr&=&\displaystyle\alpha I.
\end{array}\right.$$
Here it is assumed that the population transfer from S to I ---the disease incidence--- occurs at a rate which is proportional to the number of possible contacts of susceptible and infected individuals: $\beta SI$ for some $\beta>0$. On the other hand, the transfer from I to R ---the recovery--- is assumed to occur at a rate proportional to the number of infected individuals: $\alpha I$ for some $\alpha>0$ (Figure \ref{fig:SIR}).

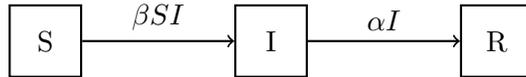
\begin{figure}
\centering
\begin{tikzpicture}
\node[rectangle,fill=white,draw=black,thick,minimum size=0.95cm] (S) at (0,0) {S};
\node[rectangle,fill=white,draw=black,thick,minimum size=0.95cm] (I) at (3,0) {I};
\node[rectangle,fill=white,draw=black,thick,minimum size=0.95cm] (R) at (6,0) {R};
\draw[->,thick] (S) edge node[above] {$\beta SI$} (I);
\draw[->,thick] (I) edge node[above] {$\alpha I$} (R);
\end{tikzpicture}
\caption{\label{fig:SIR} The compartment diagram of the SIR model.}
\end{figure}

The Kermack-McKendrick SIR model serves as a basis for more realistic SIR-type models. The latter result from various modifications, including the removal of the constant population size assumption (which implies the incorporation of birth/entrance and death/exit rates), as well as the use of alternative forms of incidence and recovery rates.

The incidence rate $\beta SI$ used in the Kermack-McKendrick SIR model is \textit{bilinear}: it increases linearly to infinity with $S$ and with $I$. In reality, as $S$ and/or $I$ become large, the incidence rate could experience inhibition, due to some change in the behaviour of susceptible individuals: they could become increasingly cautious of the spreading disease, and thus increasingly self-protective, reducing incidence. To model such a behaviour, one could employ incidence rates of the form $p SI/f(S,I)$, where $f$ is a function which increases with $S$ and with $I$. The simplest non-trivial forms are those in which $f$ is univariate and linear with a non-zero constant term: $f(S,I)=q+rI$ or $f(S,I)=q+rS$, where $q\neq 0$, which, after the rescaling $\beta = p/q$ and $\gamma=r/q$, result in the so-called \textit{saturated} incidence rates
$$\Inc_1(S,I)=\frac{\beta SI}{1+\gamma I}\qquad\text{and}\qquad\Inc_2(S,I)=\frac{\beta SI}{1+\gamma S},$$
respectively. Both rates increase with $S$ and with $I$. However, the rate $\Inc_1(S,I)$ saturates at $(\beta/\gamma)S$ for large $I$ and, since its denominator depends on $I$, has been used to model the situation whereby susceptible individuals increase their cautiousness ---and thus self-protectiveness--- \textit{as the number of infected individuals becomes large} \cite{CapassoSerio,ZhangLiu,ZhouFan,HuMaRuan,CuiQiuLiuHu,Ghoshetal,AlshammariKhan,AjbarAlqahtaniBoumaza}. By contrast, its alternative $\Inc_2(S,I)$ \cite{WeiChen,ZhangJinLiuZhang,KarBatabyal}, saturates at $(\beta/\gamma)I$ for large $S$ and has a denominator which is independent of $I$, thereby modelling the \textit{internal} cautiousness of susceptible individuals, the parameter $\gamma$ measuring the level of this cautiousness. (Functions similar to $\Inc_2(S,I)$ have also been used to model the influence of ``awareness programs'' on susceptible individuals' cautiousness; see \cite{Greenhalghetal,Dubey,ZuoLiuWang}.)

The Kermack-McKendrick SIR model also uses a \textit{linear} recovery rate: $\alpha I$, which does not take into account possible deceleration due to, e.g., the suboptimisation of hospital services. To take the latter into account, one could use recovery rates of the form
$$\Rec_1(I)=\alpha I +\frac{\delta I}{\omega +I}\qquad\text{and}\qquad \Rec_2(I)=\frac{\alpha I}{1+\rho I}.$$
Indeed, a well-accepted measure for the optimality of hospital services is the hospital bed-population ratio $\omega$ which has been incorporated to a number of models \cite{ShanZhu,CuiQiuLiuHu,AlshammariKhan,AjbarAlqahtaniBoumaza,Alqahtani} with recovery rates of the form $\Rec_1(I)$. Some other authors \cite{ZhangLiu,ZhouFan,Ghoshetal} used the alternative form $\Rec_2(I)$, which, as $I$ becomes large, increases and saturates at $\alpha/\rho$, capturing the behaviour of hospitals suboptimising services due to crowding. In populous countries, such as Indonesia, it is fitting to interpret the parameter $\rho$ ---whose large values imply low recovery rates--- as the hospitals' bed-occupancy rate (the number of hospitalised cases per isolation bed \cite{WHOI22Jul2020}); governments of such countries are struggling to maintain an ideal value of this quantity to keep health services optimal, setting up makeshift hospitals \cite{WHOI22Jul2020,WHOI23Jun2021}.

In this paper we propose a model for the spread of COVID-19 in a population which takes into account both the cautiousness level of susceptible individuals and the hospitals' bed-occupancy rate. The model is a compartmental, SIR-type model constructed based on the following assumptions.
\begin{enumerate}
\item[(i)] At any given time, each individual in the population, according to their status, belongs to exactly one of the following compartments: susceptible (S), infected (I), and recovered (R). We let $S=S(t)$, $I=I(t)$, and $R=R(t)$ be, respectively, the number of susceptible, infected, and recovered individuals at time $t\geqslant 0$, the dependence being differentiable. The size of the population at time $t$ is thus $N(t)=S(t)+I(t)+R(t)$.
\item[(ii)] \sloppy The rate of individuals entering the population ---due to births and migrations--- is a positive constant $\lambda>0$, and every newly-entered individual is susceptible. The rates of susceptible, infected, and recovered individuals exiting the population ---due to deaths--- are $\mu S$, $\left(\mu+\mu'\right) I$, and $\mu R$, respectively, where $\mu,\mu'>0$. The presence of $\mu'$ implies that infected individuals have a higher death rate than susceptible and recovered individuals.
\item[(iii)] Susceptible individuals become infected at the rate $\Inc_2(S,I)=\beta SI/(1+\gamma S)$. Here, we assume, for interpretability, that the \textit{cautiousness level} $\gamma$ of the susceptible individuals satisfies $\gamma\in[0,1]$, and that $\beta>0$.
\item[(iv)] Infected individuals become recovered at the rate $\Rec_2(I)=\alpha I/(1+\rho I)$. Here, we assume that $\rho\in[0,1]$, interpreting this parameter as the hospitals' bed-occupancy rate, and that $\alpha>0$.
\end{enumerate}
The above assumptions lead to the compartment diagram in Figure \ref{fig:presentmodel}, and thus to the following model:
\begin{equation}\label{eq:model}
\left\{\begin{array}{rcl}
\displaystyle\ds &=& \displaystyle\lambda - \mu S - \frac{\beta SI}{1+\gamma S},\\[0.4cm]
\displaystyle\di &=& \displaystyle-\mu I - \mu' I +\frac{\beta SI}{1+\gamma S}- \frac{\alpha I}{1+\rho I},\\[0.4cm]
\displaystyle\dr &=& \displaystyle-\mu R + \frac{\alpha I}{1+\rho I}.
\end{array}\right.
\end{equation}
Adding these equations, one verifies that the population size $N$ is \textit{not} constant. We shall study this model over the domain
$$\Omega:=\left\{(S(t),I(t),R(t))\in [0,\infty)^3:0 < S(t)+I(t)+R(t) \leqslant \frac{\lambda}{\mu},\,\,t\geqslant 0\right\}.$$

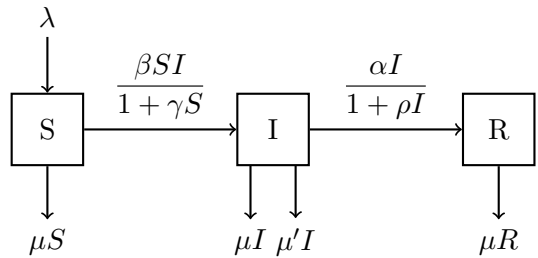
\begin{figure}
\centering
\begin{tikzpicture}
\node (muS) at (0,-1.5) {$\mu S$};
\node (muI) at (2.7,-1.5) {$\mu I$};
\node (muI') at (3.3,-1.5) {$\mu' I$};
\node (muR) at (6,-1.5) {$\mu R$};
\draw[<-,thick] (muI) edge (2.7,0);
\draw[<-,thick] (muI') edge (3.3,0);

\node[rectangle,fill=white,draw=black,thick,minimum size=0.95cm] (S) at (0,0) {S};
\node[rectangle,fill=white,draw=black,thick,minimum size=0.95cm] (I) at (3,0) {I};
\node[rectangle,fill=white,draw=black,thick,minimum size=0.95cm] (R) at (6,0) {R};
\node (lambda) at (0,1.5) {$\lambda$};
\draw[->,thick] (S) edge node[above] {$\displaystyle\frac{\beta SI}{1+\gamma S}$} (I);
\draw[->,thick] (I) edge node[above] {$\displaystyle\frac{\alpha I}{1+\rho I}$} (R);
\draw[->,thick] (lambda) edge (S);
\draw[<-,thick] (muS) edge (S);
\draw[<-,thick] (muR) edge (R);
\end{tikzpicture}
\caption{\label{fig:presentmodel} The compartment diagram of our model.}
\end{figure}

Let us now describe the organisation and main findings of this paper. In the upcoming section \ref{sec:analytic}, we study the model \eqref{eq:model} analytically. We first determine its disease-free equilibrium and basic reproduction number $\cR_0$ (subsection \ref{subsec:DFE}). Of note is the fact that $\cR_0$ depends on the cautiousness level $\gamma$ but not on the bed-occupancy rate $\rho$, meaning that, in an endemic state, the governments' effort of setting up increasingly many makeshift hospitals will never drive the system to a disease-free state if the citizens are not adequately cautious about protecting themselves from being infected. We also determine the possible numbers of its endemic equilibria (subsection \ref{subsec:DEE}), the linear stability of these equilibria and other dynamical properties of the model (subsection \ref{subsec:stability}), as well as the sensitivity index of the basic reproduction number with respect to each parameter (subsection \ref{subsec:sensitivity}). The latter confirms quantitatively that $\gamma$ is one of the parameters upon which $\cR_0$ depends most sensitively.

In section \ref{sec:numerics}, we fix a value for each of the parameters $\beta$, $\lambda$, $\mu$, $\mu'$, $\alpha$, $\rho$ and study the model \eqref{eq:model} numerically as $\gamma$ vary over its domain, presenting the results in three subsections. In subsection \ref{subsec:numerics1}, we show that the model exhibits a backward transcritical bifurcation at $\cR_0=1$, a saddle-node bifurcation, and a supercritical Hopf bifurcation whereby a stable endemic equilibrium becomes unstable and surrounded by a stable limit cycle. To discover the limit cycle's bifurcations, we use the AUTO software \cite{Doedeletal}, revealing a homoclinic bifurcation and a fold bifurcation of limit cycles. These are discussed in subsection \ref{subsec:numerics2}, where we also present plots of all qualitatively different orbital behaviours of the model corresponding to different values of $\gamma$ (Figure \ref{fig:simulations}). In subsection \ref{subsec:numerics3}, we design a scenario whereby $\gamma$ changes as a piecewise-constant function of $t$, and display the behaviour of $I$ resulting from these changes. We demonstrate that the model exhibits a hysteresis: in a near-threshold disease-free state, a small decrease of $\gamma$ could lead to an endemic state which is subsequently recoverable only by a large increase of $\gamma$. For the specified parameter values, $\gamma$ needs to exceed $0.35$ in order to guarantee the recovery.

In the final section \ref{sec:conclusions}, we conclude from our model that the pandemic cannot be resolved merely by increasing the hospitals' bed-occupancy rate; a serious effort towards a high cautiousness level of susceptible individuals is necessary. We also discuss avenues for further investigation.

\section{Analytic results}\label{sec:analytic}

Although the population size $N$ is not constant, the first two equations in our model \eqref{eq:model} do not depend on the third one; this means that the analysis of the model can be performed by considering only its first two equations,
\begin{equation}\label{eq:reducedmodel}
\left\{\begin{array}{rcl}
\displaystyle\ds &=& \displaystyle\lambda - \mu S - \frac{\beta SI}{1+\gamma S},\\[0.4cm]
\displaystyle\di &=& \displaystyle-\mu I - \mu' I +\frac{\beta SI}{1+\gamma S}- \frac{\alpha I}{1+\rho I},
\end{array}\right.
\end{equation}
over a domain $\Omega'\subseteq\R^2$ which is the projection of $\Omega$ on the $SI$-plane. In this section, we present a qualitative analysis of the reduced model \eqref{eq:reducedmodel}. This consists of, firstly, a computation of the model's disease-free equilibrium and basic reproduction number: the threshold parameter $\cR_0$ for which the disease-free equilibrium is stable if $\cR_0<1$ and unstable if $\cR_0>1$; this will be computed using the next-generation matrix approach \cite{DiekmannHeesterbeekMetz,DriesscheWatmough} (subsection \ref{subsec:DFE}). We also determine the possible numbers of the model's endemic equilibria, using Descartes' rule of signs \cite{Meserve} (subsection \ref{subsec:DEE}). Using tools from dynamical systems and bifurcation theory (see \cite{Kuznetsov,Martcheva,Robinson,Strogatz} for background), we next determine and study the linear stability of these equilibria, whether the transcritical bifurcation occurring at $\cR_0=1$ is forward or backward, and a property of a periodic orbit of our model (subsection \ref{subsec:stability}). Finally, using the sensitivity indices of the basic reproduction number \cite{ChitnisHymanCushing}, we determine the parameters upon which the basic reproduction number depends most sensitively (subsection \ref{subsec:sensitivity}).

\subsection{Disease-free equilibrium and basic reproduction number}\label{subsec:DFE}

\sloppy Letting $\be_0'=\left(S_0,0\right)\in\Omega'$ be the disease-free equilibrium of the reduced model \eqref{eq:reducedmodel} corresponding to the disease-free equilibrium $\be_0\in\Omega$ of the original model \eqref{eq:model}, the solving $\text{d}S_0/\text{d}t=0$ immediately gives $S_0=\lambda/\mu$. Thus, $$\be_0'=\left(\frac{\lambda}{\mu},0\right).$$
Since the next-generation matrix \cite[page 33]{DriesscheWatmough} for \eqref{eq:reducedmodel} is the $1\times 1$ matrix $\mathbf{F}\mathbf{V}^{-1}$, where
$$\mathbf{F}=\left[\left.\frac{\partial}{\partial I}\left(\frac{\beta SI}{1+\gamma S}\right)\right|_{(S,I)=\be_0'}\right]\quad\text{and}\quad \mathbf{V}=\left[\left.\frac{\partial}{\partial I}\left(\mu I + \mu' I + \frac{\alpha I}{1+\rho I}\right)\right|_{(S,I)=\be_0'}\right],$$
its only entry is the model's basic reproduction number:
\begin{equation}\label{eq:BRN}
\cR_0 = \frac{\beta\lambda}{(\mu+\gamma\lambda)\left(\mu+\mu'+\alpha\right)}.
\end{equation}
Notice that $\cR_0$ depends on $\gamma$ (as well as on $\beta$, $\lambda$, $\mu$, $\mu'$, and $\alpha$), but \textit{not} on $\rho$. Thus, the value of $\cR_0$ cannot be suppressed by merely reducing the hospitals' bed-occupancy rate. Ways to suppress $\cR_0$ are the following.
\begin{enumerate}
\item[(i)] Seek to achieve $\lambda=0$. In practice, this means the population's government declaring a total lockdown: outside individuals are absolutely prevented from entering the population. While such a policy is effective to eradicate the disease, it could lead to a substantial damage in the economic sector.
\item[(ii)] Seek to achieve $\beta=0$, i.e., an absolute prevention of interindividual interactions. Again, this eradicates the disease, but jeopardises the economy.
\item[(iii)] Seek to increase $\alpha$, i.e., to accelerate the treatment of infected individuals. Although possible, this is largely constrained by the limitedness of medical resources.
\end{enumerate}
Since we clearly do not wish to increase $\mu$ or $\mu'$, only one more option is available.
\begin{enumerate}[resume]
\item[(iv)] Seek to increase $\gamma$, i.e., the cautiousness of susceptible individuals.
\end{enumerate}
In subsection \ref{subsec:sensitivity}, we shall confirm quantitatively that, for the values of parameters \eqref{eq:parameters} used in our numerical simulations (section \ref{sec:numerics}), $\gamma$ is one of the parameters upon which $\cR_0$ depends most sensitively. Thus, governments of countries that are able to increase the citizens' cautiousness level (through, e.g., dissemination and enforcement of health protocols) have a reasonable chance of success in resolving the pandemic.

\bigskip\subsection{Endemic equilibria}\label{subsec:DEE}

We now characterise and determine the number of the endemic equilibria of \eqref{eq:reducedmodel}. Let $\mathbb{N}$ be the set of \textit{positive} integers, and let $\be_n'=\left(S_n,I_n\right)\in\Omega'$, where $n\in\N$, be an endemic equilibrium of \eqref{eq:reducedmodel} corresponding to an endemic equilibrium $\be_n\in\Omega$ of \eqref{eq:model}. The conditions $\text{d}S_n/\text{d}t=0$ and $\text{d}I_n/\text{d}t=0$ are equivalent to
$$\lambda - \mu S_n - \frac{\beta S_nI_n}{1+\gamma S_n}=0\qquad\text{and}\qquad -\mu I_n - \mu' I_n +\frac{\beta S_nI_n}{1+\gamma S_n}- \frac{\alpha I_n}{1+\rho I_n}=0,$$
respectively. Adding these gives
\begin{equation}\label{eq:Snendemic}
S_n=\frac{\lambda}{\mu}-I_n\left[\frac{\alpha}{\mu(1+\rho I_n)}+\frac{\mu+\mu'}{\mu}\right],
\end{equation}
which, together with the condition $\text{d}S_n/\text{d}t=0$ and the fact that $I_n\neq 0$, gives
\begin{equation}\label{eq:I*cubic}
a{I_n}^3+b{I_n}^2+cI_n+d=0,
\end{equation}
where
\begin{align*}
a &:= \rho^2\left(\mu+\mu'\right)[\left(\mu+\mu'\right)\gamma-\beta],\\
b &:= \rho^2(\mu+\gamma\lambda)\left[\left(\mu+\mu'+\alpha\right)\left(\mathcal{R}_0-1\right)+\alpha\right]+\rho\alpha\beta+2\rho\left(\mu+\mu'+\alpha\right) \left[\left(\mu+\mu'\right)\gamma-\beta\right],\\
c &:= 2\rho(\mu+\gamma\lambda)\left(\mu+\mu'+\alpha\right)\left(\mathcal{R}_0-1\right)+\left(\mu+\mu'+\alpha\right)\left[\left(\mu+\mu'\right)\gamma-\beta\right]
+\gamma \alpha\left(\mu+\mu'+\alpha\right)\\&\phantom{:=}+\rho\alpha(\mu+\gamma\lambda),\\
d &:= (\mu+\gamma\lambda)\left(\mu+\mu'+\alpha\right)\left(\mathcal{R}_0-1\right).
\end{align*}

It is straightforward to see that
\begin{eqnarray*}
\cR_0<1&\Leftrightarrow&d<0,\\
\cR_0=1&\Leftrightarrow&d=0,\\
\cR_0>1&\Leftrightarrow&d>0;
\end{eqnarray*}
that
\begin{eqnarray*}
\left(\mu+\mu'\right)\gamma-\beta<0&\Leftrightarrow&a<0,\\
\left(\mu+\mu'\right)\gamma-\beta>0&\Leftrightarrow&a>0;
\end{eqnarray*}
and that
$$\cR_0\geqslant 1\,\,\,\Rightarrow\,\,\,a<0.$$
By Descartes' rule of signs \cite[Theorem 4.10]{Meserve}, the number $\cN$ of positive roots of \eqref{eq:I*cubic}, counting multiplicities, is bounded above by the number of times the non-zero coefficients change sign, and differs from it by an even number. This gives the possible values of $\cN$ in various cases, presented in Table \ref{tab:Descartes}.

\begin{table}
\centering\renewcommand{\arraystretch}{1.4}
\scalebox{0.95}{\begin{tabular}{|l|l|l|l|}
\hline
$\cR_0$ and $d$ & $a$ & $b$ and $c$ & Possible values of $\cN$\\
\hhline{|=|=|=|=|}
$\cR_0<1\,\,\Leftrightarrow\,\,d<0$ & $a<0$ & $b<0$ and $c<0$ & $0$\\
\cline{3-4}
                                      &         & $b>0$ or $c>0$  & $0$, $2$\\
\cline{2-4}
                                      & $a>0$ & $b<0$ and $c>0$ & $1$, $3$\\
\cline{3-4}
                                      &         & $b>0$ or $c<0$  & $1$\\
\hline
$\cR_0=1\,\,\Leftrightarrow\,\,d=0$ & $a<0$ & $b<0$ and $c<0$ & $0$\\
\cline{3-4}
                                      &         & $c>0$  & $1$\\
\cline{3-4}
                                      &         & $b>0$ and $c<0$  & $2$\\
\hline
$\cR_0>1\,\,\Leftrightarrow\,\,d>0$ & $a<0$ & $b>0$ and $c<0$ & $1$, $3$\\
\cline{3-4}
                                      &         & $b<0$ or $c>0$  & $1$\\
\hline
\end{tabular}}
\caption{\label{tab:Descartes} Possible values of $\cN$ for various values of $\cR_0$ and $a$, $b$, $c$, $d$, assuming that $b$ and $c$ are non-zero.}
\end{table}

If a positive root $I_n$ of \eqref{eq:I*cubic} exists, the corresponding value of $S_n$ obtained from \eqref{eq:Snendemic} is non-negative (i.e., $\be_n'\in\Omega'$) if and only if
$$\frac{\lambda}{\mu}>I_n\left[\frac{\alpha}{\mu\left(1+\rho I_n\right)}+\frac{\mu+\mu'}{\mu}\right].$$

\bigskip\subsection{Linear stability and bifurcation analysis}\label{subsec:stability}

In this subsection, we first prove that the stability of the disease-free equilibrium is indeed determined by the basic reproduction number (Theorem \ref{prop:DFE}), and characterise the stability of the endemic equilibria (Theorem \ref{prop:DEE}). Subsequently, we establish a sufficient condition for the occurrence of forward and backward transcritical bifurcations (Theorem \ref{prop:forwardbackwardbifurcation}) and a property satisfied by a periodic orbit of our model if it exists (Theorem \ref{prop:periodicorbit}).

A stability criterion of the disease-free equilibrium in terms of the basic reproduction number is derived using the standard result that an equilibrium of a planar system is \textit{locally asymptotically stable} if both eigenvalues of the system's Jacobian evaluated at that equilibrium have negative real parts, and \textit{unstable} if at least one eigenvalue has a positive real part \cite[Theorem 4.6]{Robinson}. For completeness, we also use the term \textit{semistable} in the situation whereby the one of the eigenvalues is zero and the other is negative.\bigskip

\begin{theorem}\label{prop:DFE}
The disease-free equilibrium $\be_0$ is locally asymptotically stable if $\mathcal{R}_0<1$, semistable if $\cR_0=1$, and unstable if $\cR_0>1$.
\end{theorem}
\begin{proof}
The Jacobian of the model \eqref{eq:reducedmodel} evaluated at $(S,I)=\be_0'$, i.e.,
$$\bJ_0 = \left(\begin{array}{cc}
-\mu & -\frac{\beta\lambda}{\mu+\gamma\lambda}\\
0& -\mu-\mu'-\alpha+\frac{\beta\lambda}{\mu+\gamma\lambda}
\end{array}\right),$$
has two real eigenvalues:
$$-\mu\qquad\text{and}\qquad \frac{\beta\lambda - (\mu+\gamma\lambda)\left(\mu+\mu'+\alpha\right)}{\mu+\gamma\lambda}.$$
Since the former is negative, the disease-free equilibrium $\be_0$ is locally asymptotically stable if the latter is negative, i.e., $\mathcal{R}_0<1$, is semistable if the latter is zero, i.e., $\mathcal{R}_0=1$, and is unstable (in fact, a saddle point) if latter is positive, i.e., $\mathcal{R}_0>1$.
\end{proof}\bigskip

\noindent We shall make use of this theorem in subsection \ref{subsec:numerics1}.

Next, we turn our attention to the stability of the endemic equilibria, which will be classified using the well-known trace-determinant criterion of stability of equilibria of planar systems \cite[Theorem 4.3]{Robinson} and the Grobman-Hartman theorem \cite[Theorem 12.10]{Robinson}.\bigskip

\begin{lemma}
Let $n\in\N$, and let
\begin{equation}\label{eq:PnQn}
\begin{array}{rl}
\displaystyle\cP_n&\displaystyle:=2\mu+\mu'+\frac{\alpha}{\left(1+\rho I_n\right)^2}+\frac{\beta I_n}{\left(1+\gamma S_n\right)^2}-\frac{\beta S_n}{1+\gamma S_n},\\[0.4cm]
\displaystyle\cQ_n&\displaystyle:=\mu^2+\mu\mu'+\frac{\mu \alpha}{\left(1+\rho I_n\right)^2}+\frac{\left(\mu+\mu'\right)\beta I_n}{\left(1+\gamma S_n\right)^2}+\frac{\beta \alpha I_n}{\left(1+\gamma S_n\right)^2\left(1+\rho I_n\right)^2}-\frac{\mu\beta S_n}{1+\gamma S_n}.
\end{array}
\end{equation}
The following hold for the trivial equilibrium of the variational system of \eqref{eq:reducedmodel} near $\be_n'$.
\begin{enumerate}
\item[\textnormal{(i)}] If $\cQ_n<0$, then the equilibrium is a saddle point, and hence unstable.
\item[\textnormal{(ii)}] If $\cQ_n>0$ and $\cP_n<0$, then the fixed point is unstable; it is an unstable node if ${\cP_n}^2-4\cQ_n>0$, a degenerate unstable node if ${\cP_n}^2-4\cQ_n=0$, and an unstable focus if ${\cP_n}^2-4\cQ_n<0$.
\item[\textnormal{(iii)}] If $\cQ_n>0$ and $\cP_n>0$, then the equilibrium is asymptotically stable; it is a stable node if ${\cP_n}^2-4\cQ_n>0$, a degenerate stable node if ${\cP_n}^2-4\cQ_n=0$, and a stable focus if ${\cP_n}^2-4\cQ_n<0$.
\item[\textnormal{(iv)}] If $\cQ_n>0$ and $\cP_n=0$, then the equilibrium is $L$-stable but not asymptotically stable.
\item[\textnormal{(v)}] If $\cQ_n=0$, then the equilibrium is unstable if $\cP_n<0$, and is $L$-stable but not asymptotically stable if $\cP_n>0$.
\end{enumerate}
\end{lemma}
\begin{proof}
The Jacobian of the model \eqref{eq:reducedmodel} evaluated at $(S,I)=\be_n'$ is
$$\bJ_n=\left(\begin{array}{cc}
-\mu -\frac{\beta I_n}{1+\gamma S_n} + \frac{\beta\gamma S_n I_n}{\left(1+\gamma S_n\right)^2} & -\frac{\beta S_n}{1+\gamma S_n}\\
\frac{\beta I_n}{1+\gamma S_n}-\frac{\beta\gamma S_n I_n}{\left(1+\gamma S_n\right)^2} & -\mu-\mu'+\frac{\beta S_n}{1+\gamma S_n}-\frac{\alpha}{1+\rho I_n}+\frac{\rho\alpha I_n}{\left(1+\rho I_n\right)^2}
\end{array}\right).$$
Since $\tr\left(\bJ_n\right)=-\cP_n$ and $\det\left(\bJ_n\right)=\cQ_n$ by straightforward computation, the theorem is an immediate consequence of \cite[Theorem 4.3]{Robinson}.
\end{proof}\bigskip

By the Grobman-Hartman theorem \cite[Theorem 12.10]{Robinson}, the above lemma immediately implies the following theorem concerning the type and stability of $\be_n$.\bigskip

\begin{theorem}\label{prop:DEE}
Let $n\in\N$.
\begin{enumerate}
\item[\textnormal{(i)}] If $\cQ_n<0$, then $\be_n$ is a saddle point, and hence unstable.
\item[\textnormal{(ii)}] If $\cQ_n>0$ and $\cP_n<0$, then $\be_n$ is unstable; it is an unstable node if ${\cP_n}^2-4\cQ_n>0$, and an unstable focus if ${\cP_n}^2-4\cQ_n<0$.
\item[\textnormal{(iii)}] If $\cQ_n>0$ and $\cP_n>0$, then $\be_n$ is asymptotically stable; it is a stable node if ${\cP_n}^2-4\cQ_n>0$, and a stable focus if ${\cP_n}^2-4\cQ_n<0$.
\end{enumerate}
\end{theorem}\bigskip

\noindent This theorem will also be applied in subsection \ref{subsec:numerics1}.

In the case of $\cR_0<1$, the disease-free equilibrium $\be_0$, being stable, may or may not be the only existing equilibrium of the model \eqref{eq:reducedmodel}. This depends on whether the model exhibits a \textit{forward} or a \textit{backward transcritical} bifurcation at $\cR_0=1$, i.e., whether $\partial I_n/\partial\cR_0$ is positive or negative at $\left(\cR_0,I_n\right)=(1,0)$ \cite[section 7.5]{Martcheva}. In the case of a backward bifurcation, besides the stable disease-free equilibrium, an unstable endemic equilibrium exists for $\cR_0<1$. We derive sufficient conditions for these bifurcations.\bigskip

\begin{theorem}\label{prop:forwardbackwardbifurcation}
At $\cR_0=1$, the model \eqref{eq:reducedmodel} exhibits a forward bifurcation if $$\beta<\frac{\mu\left(\mu+\mu'+\alpha\right)^3}{\alpha\lambda^2\rho},$$
and a backward bifurcation if
$$\beta>\frac{\mu\left(\mu+\mu'+\alpha\right)^3}{\alpha\lambda^2\rho}.$$
\end{theorem}
\begin{proof}
From \eqref{eq:BRN} one obtains
$$\gamma=\frac{\beta}{\cR_0\left(\mu+\mu'+\alpha\right)}-\frac{\mu}{\lambda}.$$
Substituting this into \eqref{eq:I*cubic} and differentiating both sides with respect to $\cR_0$, one obtains that
$$\left.\frac{\partial I_n}{\partial\cR_0}\right|_{\left(\cR_0,I_n\right)=(1,0)}=\frac{\beta\lambda^2\left(\mu+\mu'+\alpha\right)}{\mu\left(\mu+\mu'+\alpha\right)^3-\alpha\beta\lambda^2\rho}.$$
A forward (backward, respectively) bifurcation occurs if this quantity is positive (negative, respectively), proving the theorem.
\end{proof}\bigskip

Our final theorem is an application of Dulac criterion \cite[Theorem 6.11]{Robinson} to describe a property of a periodic orbit of our reduced model \eqref{eq:reducedmodel} upon its existence: such an orbit must intersect a specific quadratic curve in $S$ and $I$ whose formula depends on the model's parameters. For a description of the many possible shapes of such a curve, see, e.g., \cite{Kesavan}. For the values of parameters \eqref{eq:parameters} used in section \ref{sec:numerics}, the curve is a hyperbola.\bigskip

\begin{theorem}\label{prop:periodicorbit}
A periodic orbit of the model \eqref{eq:reducedmodel}, if it exists, intersects the curve
\begin{equation}\label{eq:curve}
\tilde{a}I^2+\tilde{b}SI+\tilde{c}S+\tilde{d}I+\tilde{e}=0,
\end{equation}
where
\begin{align*}
\tilde{a}&:=-\beta\rho,\\
\tilde{b}&:=-4\gamma\mu\rho-2\gamma\mu'\rho+2\beta\rho,\\
\tilde{c}&:=-\alpha\gamma-3\gamma\mu-\gamma\mu'+\beta,\\
\tilde{d}&:=\gamma\lambda\rho-3\mu\rho-2\mu'\rho-\beta,\\
\tilde{e}&:=\gamma\lambda-\alpha-2\mu-\mu'.
\end{align*}
\end{theorem}
\begin{proof}
For every $(S,I)\in\Omega'$, let
$$g(S,I):=(1+\gamma S)(1+\rho I),$$
and let $f_1(S,I)$ and $f_2(S,I)$ be the right-hand sides of the equations in \eqref{eq:reducedmodel}. Direct computation shows that
$$\frac{\partial (g(S,I)f_1(S,I))}{\partial S}+\frac{\partial (g(S,I)f_2(S,I))}{\partial I}$$
is precisely the left-hand side of \eqref{eq:curve}. By Dulac criterion, this means that there cannot be any periodic orbit contained entirely in
$$\left\{(S,I)\in\Omega':\tilde{a}I^2+\tilde{b}SI+\tilde{c}S+\tilde{d}I+\tilde{e}<0\right\}$$
or entirely in
$$\left\{(S,I)\in\Omega':\tilde{a}I^2+\tilde{b}SI+\tilde{c}S+\tilde{d}I+\tilde{e}>0\right\}.$$
The theorem follows.
\end{proof}\bigskip

\noindent We will illustrate this theorem for the cases considered in section \ref{sec:numerics} in which periodic orbits exist.

\smallskip\subsection{Sensitivity analysis}\label{subsec:sensitivity}

In subsection \ref{subsec:DFE} we have discussed how the value of $\cR_0$ can be suppressed by changing the values of the parameters on which it depends: $\beta$, $\lambda$, $\gamma$, $\mu$, $\mu'$, and $\alpha$. Let us now discuss a quantity which measures the relative impact of each of these parameters on $\cR_0$. The \textit{sensitivity index} of the basic reproduction number $\cR_0$ to a parameter $p\in\left\{\beta,\lambda,\gamma,\mu,\mu',\alpha\right\}$ is given by
$$\SI_p:=\frac{\partial\cR_0}{\partial p}\cdot\frac{p}{\cR_0},$$
i.e., the ratio of the relative change in $\cR_0$ to the relative change in $p$, assuming the required differentiability \cite{ChitnisHymanCushing}. Direct computation gives
\begin{center}
\begin{tabular}{rclcrcl}
$\displaystyle\SI_\beta$&=&$\displaystyle1$, &\phantom{aaaaaa}& $\displaystyle\SI_\mu$&=&$\displaystyle-\frac{\mu (2\mu+\mu'+\alpha+\gamma\lambda)}{\left(\mu+\mu'+\alpha\right)(\mu+\gamma\lambda)}$,\\[0.4cm]
$\displaystyle\SI_\lambda$&=&$\displaystyle\frac{\mu}{\mu+\gamma\lambda}$, && $\displaystyle\SI_{\mu'}$ &=& $\displaystyle-\frac{\mu'}{\mu+\mu'+\alpha}$,\\[0.4cm]
$\displaystyle\SI_\gamma$&=&$\displaystyle-\frac{\gamma\lambda}{\mu+\gamma\lambda}$, && $\displaystyle\SI_\alpha$&=&$\displaystyle-\frac{\alpha}{\mu+\mu'+\alpha}$.
\end{tabular}
\end{center}
Notice that $\SI_\beta$, $\SI_{\mu'}$, and $\SI_\alpha$ are independent of $\gamma$.

For the values of parameters \eqref{eq:parameters} used on our numerical simulations (section \ref{sec:numerics}), the values of these indices in the cases of $\gamma=0.3497$ (where $\cR_0<1$) and $\gamma=0.1$ (where $\cR_0>1$) are presented in Table \ref{tab:sensitivity}. In both cases we can see that the parameters most sensitive to $\cR_0$, i.e., those with the largest sensitivity indices in absolute value, are $\beta$ and $\gamma$. Indeed, a 1\% increase in the parameter $\beta$ ($\gamma$, respectively) results in a $1\%$ increase ($0.997\%$ decrease, respectively) in the basic reproduction number. This confirms that the governments' effort to resolve the pandemic by increasing the citizens' cautiousness level is reasonable (cf.\ subsection \ref{subsec:DFE}).

\begin{table}
\centering\renewcommand{\arraystretch}{1.4}
\begin{tabular}{|l|l|l|}
\hline
\multirow{2}{*}{Sensitivity index} & $\gamma=0.3497$,   & $\gamma=0.1$,\\
                                   & $\cR_0=0.4599<1$ & $\cR_0=1.5970>1$\\
\hhline{|=|=|=|}
$\SI_\beta$ & $1$ & $1$ \\\hline
$\SI_\lambda$ & $0.002851$ & $0.009901$\\\hline
$\SI_\gamma$ & $-0.99715$ & $-0.9901$\\\hline
$\SI_\mu$ & $-0.03511$ & $-0.04216$\\\hline
$\SI_{\mu'}$ & $-0.32258$ & $-0.32258$\\\hline
$\SI_\alpha$ & $-0.64516$ & $-0.64516$\\\hline
\end{tabular}
\caption{\label{tab:sensitivity} Values of the sensitivity indices of the basic reproduction number of our model for $\beta=0.05$, $\lambda=10$, $\mu=0.01$, $\mu'=0.1$, $\alpha=0.2$, and $\gamma$ as in the first row.}
\end{table}

\section{Numerical results}\label{sec:numerics}

In this section, we present the results of our numerical explorations. These are carried out using the following parameter values, which, as we shall see, are chosen to expose the rich dynamical behaviour of the model as $\gamma$ varies over its domain\footnote{Certainly, the parameter values can also be chosen or estimated in order to describe the epidemic situation in a particular region. Indeed, we have used the same model to describe the situation in Jakarta over a specific period, whereby the values of $\lambda$, $\mu$, $\mu'$, and $\rho$ are adopted from \cite{AldilaEtAl,Kemkes}, while those of $\alpha$ and $\beta$ are estimated via discretisation and the so-called L-BFGS-B algorithm \cite{FeiRongWangWang} using the data provided by the Johns Hopkins Coronavirus Resource Center \cite{JHU}; see \cite{YongHoseanaOwen} for details. The same algorithm can also be used to estimate $\gamma$.}:
\begin{equation}\label{eq:parameters}
\beta=0.05,\quad\lambda=10,\quad\mu=0.01,\quad\mu'=0.1,\quad\alpha=0.2,\quad\text{and}\quad\rho=0.1.
\end{equation}
We divide the presentation into three subsections. In the first subsection, we analyse the model by applying our analytic results in section \ref{sec:analytic}; this results in the finding of a \textit{backward transcritical bifurcation} at $\gamma\approx 0.1602903226$, a \textit{Hopf bifurcation} at $\gamma\approx 0.3496375754$, and a \textit{saddle-node bifurcation} at $\gamma\approx 0.3569024925$. In the second subsection, we use the software AUTO to discover bifurcations which are undetectable analytically: a \textit{homoclinic bifurcation} at $\gamma\approx 0.3498971211$ and a \textit{fold bifurcation of a limit cycle} at $\gamma\approx 0.3500585184$. In the final subsection, we construct a scenario whereby the susceptible individuals' cautiousness level $\gamma$ varies over time, displaying the behaviour of $I$ over time and showing that our model exhibits a \textit{hysteresis}.

\subsection{Use of analytic results}\label{subsec:numerics1}

For the values of parameters in \eqref{eq:parameters}, the basic reproduction number \eqref{eq:BRN} of our model \eqref{eq:model} reads 
\begin{equation}\label{eq:parametersR0}
\cR_0=\frac{5000}{31+31000\gamma},
\end{equation}
which decreases monotonically with $\gamma$, and our disease-free equilibrium is $\be_0=(\lambda/\mu,0,0)=(1000,0,0)$. This equilibrium exists for all values of $\gamma$, and is stable if $\gamma>\gamma^{(\TR)}$, semistable if $\gamma=\gamma^{(\TR)}$, and unstable if $\gamma<\gamma^{(\TR)}$, where $\gamma^{(\TR)}=4969/31000$ (Theorem \ref{prop:DFE}). The coefficients of \eqref{eq:I*cubic} are
\begin{equation}\label{eq:parametersai}
\begin{array}{rl}
a&\displaystyle=\frac{121}{1000000}\gamma-\frac{11}{200000}\begin{cases}
<0,&\text{if }\gamma<\gamma^{(a)};\\
=0,&\text{if }\gamma=\gamma^{(a)};\\
>0,&\text{if }\gamma>\gamma^{(a)},
\end{cases}\\
b&\displaystyle=-\frac{209}{50000}\gamma+\frac{2889}{1000000}\begin{cases}
>0,&\text{if }\gamma<\gamma^{(b)};\\
=0,&\text{if }\gamma=\gamma^{(b)};\\
<0,&\text{if }\gamma>\gamma^{(b)},
\end{cases}\\
c&\displaystyle=-\frac{3239}{10000}\gamma+\frac{1051}{12500}\begin{cases}
>0,&\text{if }\gamma<\gamma^{(c)};\\
=0,&\text{if }\gamma=\gamma^{(c)};\\
<0,&\text{if }\gamma>\gamma^{(c)},
\end{cases}\\
d&\displaystyle=-\frac{31}{10}\gamma+\frac{4969}{10000}\begin{cases}
>0,&\text{if }\gamma<\gamma^{(\TR)};\\
=0,&\text{if }\gamma=\gamma^{(\TR)};\\
<0,&\text{if }\gamma>\gamma^{(\TR)},
\end{cases}
\end{array}
\end{equation}
where
$$\gamma^{(a)}=\frac{5}{11},\qquad\gamma^{(b)}=\frac{2889}{4180},\qquad\text{and}\qquad \gamma^{(c)}=\frac{4204}{16195}$$
satisfy
$$\gamma^{(\TR)}<\gamma^{(c)}<\gamma^{(a)}<\gamma^{(b)}.$$
By Descartes' rule of signs (cf.\ subsection \ref{subsec:DEE}), this implies that
$$\cN\in\begin{cases}
\{1\},&\text{if }0\leqslant \gamma\leqslant \gamma^{(\TR)}\text{ or }\gamma^{(a)}\leqslant\gamma\leqslant 1;\\
\{0,2\},&\text{if }\gamma^{(\TR)}<\gamma<\gamma^{(a)}.
\end{cases}$$

To find the specific value of $\cN$ in each possible case, we perform explicit computation: substituting \eqref{eq:parametersai} into \eqref{eq:I*cubic} and solving for $\gamma$ give
\begin{equation}\label{eq:parametersgamma}
\gamma=\frac{\left(I_n+10\right)\left(55{I_n}^2-3439I_n-49690\right)}{\left(11I_n+310\right)\left(11{I_n}^2-690I_n-10000\right)}.
\end{equation}
Notice in particular the singularity
$$I^{(s)}=\frac{345}{11}+\frac{5}{11}\sqrt{9161},$$
which is the larger root of the quadratic polynomial $11{I_n}^2-690I_n-10000$.

Next, substituting \eqref{eq:parameters} into \eqref{eq:Snendemic} gives
\begin{equation}\label{eq:parametersSn}
S_n=-\frac{11{I_n}^2-690I_n-10000}{I_n+10}\begin{cases}
>0,&\text{if }I_n<I^{(s)};\\
\leqslant 0,&\text{if }I_n\geqslant I^{(s)},
\end{cases}
\end{equation}
meaning that the endemic equilibrium exists if and only if $I_n<I^{(s)}$. Consequently, we restrict our attention the region $\left[0,\gamma^{(\SN)}\right]\times\left[0,I^{(s)}\right)$, where $\left(I^{(\SN)},\gamma^{(\SN)}\right)\approx(65.1955050073,0.3569024925)$ is the maximum point of the right-hand side of \eqref{eq:parametersgamma} for $0\leqslant I_n<I^{(s)}$. The bifurcation diagram consisting of points $\left(\gamma,I_n\right)$ in this region satisfying \eqref{eq:parametersgamma} is plotted in Figure \ref{fig:bifurcationdiagram} (left).

\begin{figure}\centering
\begin{tikzpicture}
\begin{axis}[
	xmin=0,
	xmax=0.43,
	ymin=0,
	ymax=83,
    xtick={0.1602903226,.3496365152,.3569024928},
    xticklabels={$\gamma^{(\TR)}$,$\gamma^{(\HB)}\,\,\,\,\,\,\,\,\,\,\,\,$,$\,\,\,\,\,\,\,\,\,\,\,\,\gamma^{(\SN)}$},    
    ytick={65.19550517,70.72118905,74.86959682},
    yticklabels={$I^{(\SN)}$,$I^{(\HB)}$,$I^{(s)}$},
	samples=100,
	xlabel=$\gamma$,
	ylabel=$I_n$,
	width=0.475\textwidth,
    height=0.475\textwidth,
    axis lines=middle,
    x axis line style=->,
    y axis line style=->,
    clip=false,
    tick label style={font=\scriptsize},
]
\addplot [domain = 0:70.72118905, samples = 300, very thick, red, dashed]
      ({(.4545454545*(x-74.63262880))*(x+10.00000002)*(x+12.10535604)/((x-74.86959682)*(x+12.14232409)*(x+28.18181818))}, {x});
\addplot [domain = 70.72118905:74.63262879, samples = 300, very thick, red]
      ({(.4545454545*(x-74.63262880))*(x+10.00000002)*(x+12.10535604)/((x-74.86959682)*(x+12.14232409)*(x+28.18181818))}, {x});
\addplot [domain = 0:0.1602903226, samples = 300, very thick, blue, dashed]
      ({x}, {0});
\addplot [domain = 0.1602903226:0.42, samples = 300, very thick, blue]
      ({x}, {0});
\draw[dotted] (axis cs:.1602903226,0) -- (axis cs:.1602903226,79);
\draw[dotted] (axis cs:.3496365152,0) -- (axis cs:.3496365152,79);
\draw[dotted] (axis cs:.3569024928,0) -- (axis cs:.3569024928,79);
\node[blue,above] at (axis cs:0.2975,0) {$I_0$};
\node[red] at (axis cs:0.2975,34) {$I_2$};
\node[red] at (axis cs:0.3082443924,77) {$I_1$};
\draw[dotted] (axis cs:0,74.86959682)--(axis cs:0.42,74.86959682);
\draw[dotted] (axis cs:0,70.72118905)--(axis cs:0.42,70.72118905);
\draw[dotted] (axis cs:0,65.19550517)--(axis cs:0.42,65.19550517);
\end{axis}
\end{tikzpicture}\qquad\begin{tikzpicture}
\begin{axis}[
	xmin=0.1611291934,
	xmax=1.19,
	ymin=0,
	ymax=83,
    xtick={0.4506543714,0.4599929432,1},
    xticklabels={$\cR_0^{(\SN)}\,\,\,\,\,\,\,\,\,\,\,\,\,\,\,$,$\,\,\,\,\,\,\,\,\,\,\,\,\,\,\,\cR_0^{(\HB)}$,$1$},    
    ytick={65.19550517,70.72118905,74.86959682},
    yticklabels={$I^{(\SN)}$,$I^{(\HB)}$,$I^{(s)}$},
	samples=100,
	xlabel=$\cR_0$,
	ylabel=$I_n$,
	width=0.475\textwidth,
    height=0.475\textwidth,
    axis lines=middle,
    x axis line style=->,
    y axis line style=->,
    clip=false,
    tick label style={font=\scriptsize},
]
\addplot [domain = 0:70.72118905, samples = 300, very thick, red, dashed]
      ({(.3540597781*(x-74.86959684))*(x+12.14232410)*(x+28.18181816)/((x-74.63326071)*(x+10.04073870)*(x+12.10472242))}, {x});
\addplot [domain = 70.72118905:74.49199531, samples = 300, very thick, red, ->]
      ({(.3540597781*(x-74.86959684))*(x+12.14232410)*(x+28.18181816)/((x-74.63326071)*(x+10.04073870)*(x+12.10472242))}, {x});
\addplot [domain = 1:1.15, samples = 300, very thick, blue, dashed]
      ({x}, {0});
\addplot [domain = 0.1611291934:1, samples = 300, very thick, blue]
      ({x}, {0});
\draw[dotted] (axis cs:0.1611291934,74.86959682)--(axis cs:1.15,74.86959682);
\draw[dotted] (axis cs:0.1611291934,70.72118905)--(axis cs:1.15,70.72118905);
\draw[dotted] (axis cs:0.1611291934,65.19550517)--(axis cs:1.15,65.19550517);
\draw[dotted] (axis cs:0.4506543714,0)--(axis cs:0.4506543714,83);
\draw[dotted] (axis cs:0.4599929432,0)--(axis cs:0.4599929432,83);
\node[blue,above] at (axis cs:0.55,0) {$I_0$};
\node[red] at (axis cs:0.55,34) {$I_2$};
\node[red] at (axis cs:0.58,77) {$I_1$};
\end{axis}
\end{tikzpicture}
\caption{\label{fig:bifurcationdiagram} On the left panel, plot of points $\left(\gamma,I_n\right)\in\left[0,\gamma^{(\SN)}\right]\times\left[0,I^{(s)}\right)$ satisfying \eqref{eq:parametersgamma} (red) and the line $I_0=0$ representing the disease-free equilibrium (blue). On the right panel, plot of points $\left(\cR_0,I_n\right)$ satisfying \eqref{eq:parametersR02} (red) and the line $I_0=0$ representing the disease-free equilibrium (blue). Solid and dashed lines indicate stability and instability, respectively.}
\end{figure}
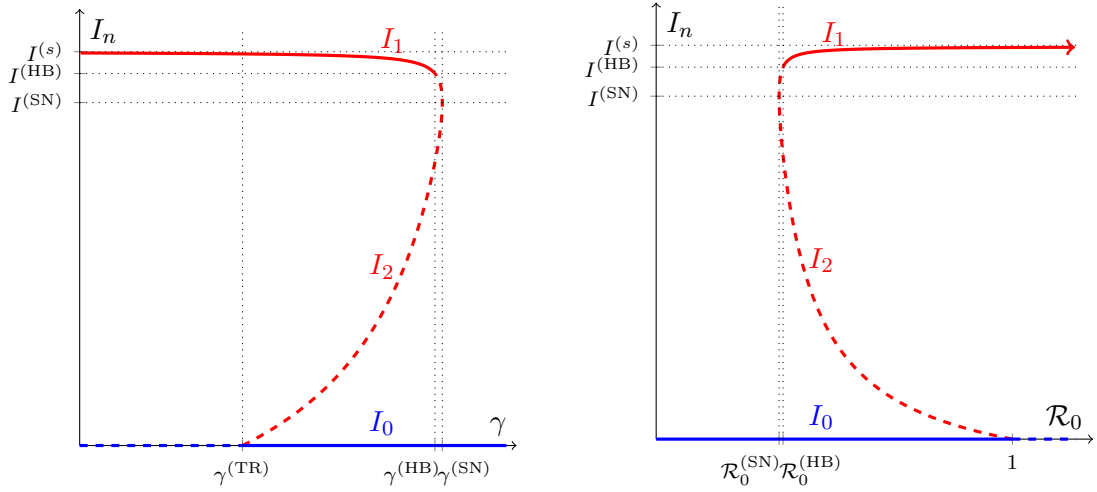

The stabilities of the endemic equilibria displayed in Figure \ref{fig:bifurcationdiagram} (left) are obtained as follows. Substituting \eqref{eq:parametersgamma} and \eqref{eq:parametersSn} into \eqref{eq:PnQn}, one obtains $\cP_n$, $\cQ_n$, and ${\cP_n}^2-4\cQ_n$ as functions of $I_n$ and reveals that, for $0<I_n<I^{(s)}$,
$$\cQ_n\begin{cases}
<0,&\text{if }0<I_n<I^{(\SN)};\\
=0,&\text{if }I_n=I^{(\SN)};\\
>0,&\text{if }I^{(\SN)}<I_n<I^{(s)}
\end{cases}\qquad\text{and}\qquad {\cP_n}^2-4\cQ_n\begin{cases}
>0,&\text{if }0<I_n<I^{(r)};\\
=0,&\text{if }I_n=I^{(r)};\\
<0,&\text{if }I^{(r)}<I_n<I^{(q)};\\
=0,&\text{if }I_n=I^{(q)};\\
>0,&\text{if }I^{(q)}<I_n<I^{(s)};\\
\end{cases}$$
and, for $I^{(\SN)}\leqslant I_n<I^{(s)}$,
$$\cP_n\begin{cases}
<0,&\text{if }I^{(\SN)}\leqslant I_n<I^{(\HB)};\\
=0,&\text{if }I_n=I^{(\HB)};\\
>0,&\text{if }I^{(\HB)}<I_n<I^{(s)},
\end{cases}$$
where $I^{(r)}\approx 65.6956172723$, $I^{(q)}\approx 74.2040760587$, and $I^{(\HB)}\approx 70.7209428218$ satisfy
$$0<I^{(\SN)}<I^{(r)}<I^{(\HB)}<I^{(q)}<I^{(s)}.$$
Letting the pair $\left(I^{(\HB)},\gamma^{(\HB)}\right)$ satisfy \eqref{eq:parametersgamma}, we have $\gamma^{(\HB)}\approx 0.3496375754$. These suffice to deduce the stability of our endemic equilibria as shown in Figure \ref{fig:bifurcationdiagram} (left), using Theorem \ref{prop:DEE}:
\begin{itemize}
\item The equilibrium $\be_1$, existing for $0\leqslant\gamma\leqslant\gamma^{(\SN)}$, is locally asymptotically stable for $0\leqslant\gamma<\gamma^{(\HB)}$, and unstable\footnote{The equilibrium $\be_1'$ of \eqref{eq:reducedmodel} is an unstable focus for $\gamma^{(\HB)}<\gamma<\gamma^{(r)}$, and an unstable node for $\gamma^{(r)}<\gamma<\gamma^{(\SN)}$, where $\gamma^{(r)}\approx 0.3568741376$ is such that the pair $\left(I^{(r)},\gamma^{(r)}\right)$ satisfies \eqref{eq:parametersgamma}.} for $\gamma^{(\HB)}<\gamma\leqslant\gamma^{(\SN)}$ (coinciding with $\be_2$ for $\gamma=\gamma^{(\SN)}$).
\item The equilibrium $\be_2$, existing for $\gamma^{(\TR)}\leqslant\gamma\leqslant\gamma^{(\SN)}$, is semistable (coinciding with $\be_0$) for $\gamma=\gamma^{(\TR)}$, and unstable\footnote{The equilibrium $\be_2'$ of \eqref{eq:reducedmodel} is a saddle point for $\gamma^{(\TR)}<\gamma<\gamma^{(\SN)}$.} for $\gamma^{(\TR)}<\gamma\leqslant\gamma^{(\SN)}$ (coinciding with $\be_1$ for $\gamma=\gamma^{(\SN)}$).
\end{itemize}

Therefore, at $\gamma=\gamma^{(\SN)}$, a \textit{saddle-node bifurcation} \cite[section 8.1]{Strogatz} occurs: the two unstable endemic equilibria $\be_1$ and $\be_2$ approach each other, coalesce, and disappear. At $\gamma=\gamma^{(\TR)}$, our model undergoes a \textit{transcritical bifurcation} \cite[section 8.1]{Strogatz}: the disease-free equilibrium $\be_0$ and an appearing endemic equilibrium $\be_2$ exchange stability. Since the unstable endemic equilibrium $\be_2$ exists for $\left|\gamma-\gamma^{(\TR)}\right|$ small and $\gamma>\gamma^{(\TR)}$ (Figure \ref{fig:bifurcationdiagram} (left)), we can see that the transcritical bifurcation is \textit{backward}. This is agrees with the fact that
$$\frac{\mu\left(\mu+\mu'+\alpha\right)^3}{\alpha\lambda^2\rho}=\frac{29791}{200000000}<\beta,$$
as prescribed by Theorem \ref{prop:forwardbackwardbifurcation}. For a diagram on the $\cR_0I_n$-plane, we first let
$$\cR_0^{(\SN)}=\frac{5000}{31+31000\gamma^{(\SN)}}\approx 0.4506543710$$
and
$$\cR_0^{(\HB)}=\frac{5000}{31+31000\gamma^{(\HB)}}\approx 0.4599915523.$$
From \eqref{eq:parametersR0} we obtain
$$\gamma=\frac{5000-31\cR_0}{31000\cR_0}.$$
Substituting this into \eqref{eq:parametersai}, and then \eqref{eq:parametersai} into \eqref{eq:I*cubic}, gives
\begin{equation}\label{eq:parametersR02}
\cR_0=\frac{5000\left(121{I_n}^3-4180{I_n}^2-323900I_n-3100000\right)}{31\left(55121{I_n}^3-2893180{I_n}^2-84403900I_n-500000000\right)}.
\end{equation}
The bifurcation diagram consisting of points $\left(\cR_0,I_n\right)$ satisfying this equation is plotted in Figure \ref{fig:bifurcationdiagram} (right), showing that there are two endemic equilibria $\be_1$, $\be_2$ for $\cR_0^{(\SN)}<\cR_0<1$ and one endemic equilibrium $\be_1$ for $\cR_0\geqslant 1$.

Next, for $n=1$, substituting \eqref{eq:parametersgamma} and \eqref{eq:parametersSn} into the first equation in \eqref{eq:PnQn}, one obtains an expression for $\cP_1$ in terms of $I_1$. Differentiating this with respect to $I_1$ and evaluating the result at $I_1=I^{(\HB)}$ gives
$$\left.\frac{\partial\cP_1}{\partial I_1}\right|_{I_1=I^{(\HB)}}\approx 0.0059945065.$$
On the other hand, differentiating the right-hand side of \eqref{eq:parametersgamma} with respect to $I_1$ and evaluating the result at $I_1=I^{(\HB)}$ gives
$$\left.\frac{\partial\gamma}{\partial I_1}\right|_{I_1=I^{(\HB)}}\approx -0.0043073268.$$
The derivative of the common real part of the eigenvalues of the equilibrium $\be_1'$ of \eqref{eq:reducedmodel}, namely $-\cP_1/2$, with respect to $\gamma$, satisfies
$$\frac{\partial \left(-\cP_1/2\right)}{\partial \gamma}=\frac{\partial \left(-\cP_1/2\right)}{\partial I_1}\cdot\frac{\partial I_1}{\partial \gamma}=-\frac{1}{2}\cdot\frac{\partial \cP_1}{\partial I_1}\cdot\frac{1}{\partial \gamma/\partial I_1};$$
its value at $\gamma=\gamma^{(\HB)}$ is thus positive. Since $\cP_1=0$ and $\cQ_1\neq 0$ at $\gamma=\gamma^{(\HB)}$, this implies that our model undergoes a \textit{supercritical Hopf bifurcation} \cite[Theorem 6.6]{Robinson} at $\gamma=\gamma^{(\HB)}$: the endemic equilibrium $\be_1$ changes from a stable equilibrium (for $\left|\gamma-\gamma^{(\HB)}\right|$ small and $\gamma<\gamma^{(\HB)}$) into an unstable focus surrounded by a stable limit cycle (for $\left|\gamma-\gamma^{(\HB)}\right|$ small and $\gamma>\gamma^{(\HB)}$).

However, neither the cycle's equation, nor its behaviour as $\gamma$ is increased, is easy to obtain analytically. For this reason, we resort to a more sophisticated computer assistance. In the next subsection, we use AUTO, a numerical continuation and bifurcation software for ordinary differential equations, to carry out further explorations in this direction, revealing two bifurcations which have not been found analytically.

\subsection{Use of AUTO}\label{subsec:numerics2}

For the parameter values in \eqref{eq:parameters} and $\gamma = 0.1$, we first compute the endemic equilibrium $\be_1$. Using AUTO, we then follow this equilibrium as we vary $\gamma$. Three bifurcations are found: a transcritical bifurcation at $\gamma = \gamma^{(\TR)}$, a Hopf bifurcation at $\gamma =  \gamma^{(\HB)}$, and a saddle node bifurcation at $\gamma = \gamma^{(\SN)}$, as expected. From $\gamma =  \gamma^{(\HB)}$, we follow a periodic solution and obtain a plot of the period of the model's periodic orbits versus $\gamma$ (Figure \ref{fig:period}). The plot reveals the existence of two new critical values, $\gamma^{(\HM)}\approx 0.3498971211$ and $\gamma^{(\FLC)}\approx 0.3500585184$, where our model undergoes a \textit{homoclinic bifurcation} and a \textit{fold bifurcation of a limit cycle}, respectively. A complete summary of the number and stability of our model's equilibria and periodic orbits for all possible values of $\gamma$ is presented in Table \ref{tab:summary}. 

\begin{figure}\centering
\includegraphics[width=0.475\textwidth]{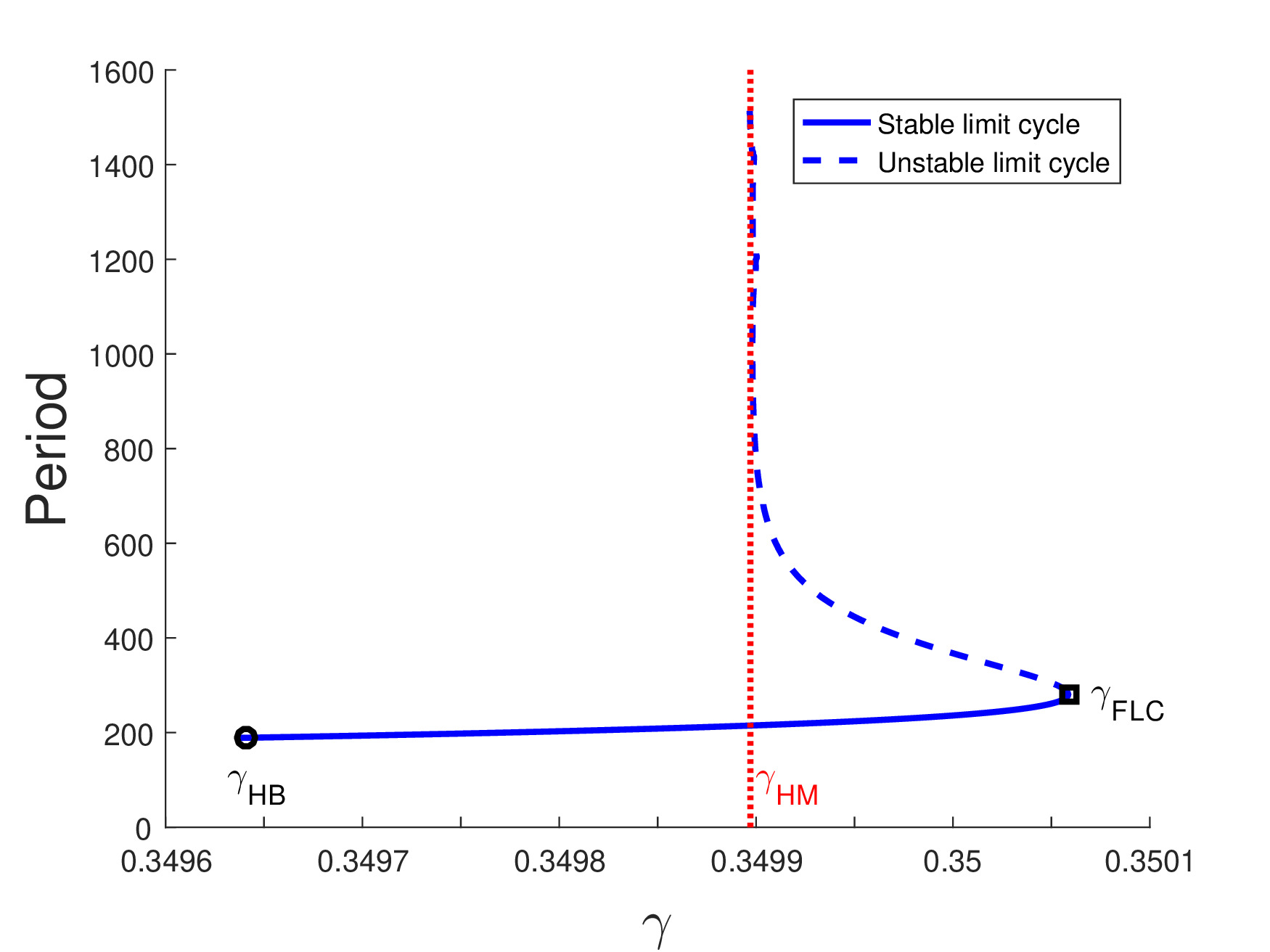}
\caption{\label{fig:period} The periodic orbits' periods against the parameter $\gamma$.}
\end{figure}

\begin{table}
\centering\renewcommand{\arraystretch}{1.4}
\scalebox{0.775}{\begin{tabular}{|c|c|c|c|c|}\hline
	\multirow{2}{*}{Case} & \multirow{2}{*}{Range of $\gamma$} & \multicolumn{3}{c|}{Number and stability of}\\\cline{3-5}
	&& Disease-free equilibrium & Endemic equilibria & Periodic orbits\\
	 \hhline{|=|=|=|=|=|}
	I & $0 \leqslant \gamma < \gamma^{(\TR)}$ & 1 unstable ($\be_0$) & 1 stable ($\be_1$) & 0\\\hline
	II & $\gamma = \gamma^{(\TR)}$ & 1  semistable ($\be_0$) & 1 stable ($\be_1$), 1 semistable ($\be_2$) & 0\\\hline
	III & $\gamma^{(\TR)} < \gamma \leqslant \gamma^{(\HB)}$ & \multirow{8}{*}{1 stable ($\be_0$)} & 1 stable ($\be_1$), 1 unstable ($\be_2$) & 0 \\\cline{1-2}\cline{4-5}
	IV & $\gamma^{(\HB)} < \gamma < \gamma^{(\HM)}$ &   & \multirow{5}{*}{\begin{minipage}{3.8cm}\centering 1 unstable ($\be_1$), 1 unstable ($\be_2$)\end{minipage}} & 1 stable\\\cline{1-2}\cline{5-5}
	V & $ \gamma = \gamma^{(\HM)}$ &    &     & 1 stable, 1 homoclinic orbit\\\cline{1-2}\cline{5-5}
	VI & $\gamma^{(\HM)} < \gamma < \gamma^{(\FLC)}$ &    &     & 1 stable, 1 unstable\\\cline{1-2}\cline{5-5}
	VII & $\gamma= \gamma^{(\FLC)}$ &    &     & 1 semistable\\\cline{1-2}\cline{5-5}
	VIII & $\gamma^{(\FLC)} < \gamma < \gamma^{(\SN)}$ &    &      & 0 \\\cline{1-2}\cline{4-5}
	IX & $ \gamma = \gamma^{(\SN)}$ &    & 1 unstable ($\be_1=\be_2$) & 0\\\cline{1-2}\cline{4-5}
	X & $ \gamma^{(\SN)}<\gamma \leqslant 1$ &    & 0 & 0\\\hline
\end{tabular}}
\caption{\label{tab:summary} Number and stability of disease-free equilibrium, endemic equilibria, and limit cycles of our model for all possible values of $\gamma$.}
\end{table}

In each of the ten cases considered in Table \ref{tab:summary}, we have generated a phase portrait of our model \eqref{eq:model}, simulating and plotting in the $SIR$-space orbits corresponding to a number of initial conditions. The results are presented in Figure \ref{fig:simulations}, which we now describe. We begin with small values of $\gamma$. In cases I and II, $0\leqslant\gamma<\gamma^{(\TR)}$ and $\gamma = \gamma^{(\TR)}$, respectively, we can see that orbits approach the only stable equilibrium: the endemic equilibrium $\be_1$ (panels (a) and (b)). For $\gamma>\gamma^{(\TR)}$, the disease-free equilibrium $\be_0$ is also stable. In case III, therefore, orbits approach either $\be_1$ or $\be_0$ (panel (c)).

In cases IV to VIII, $\gamma^{(\HB)}<\gamma<\gamma^{(\SN)}$, both endemic equilibria are unstable. In case IV, orbits approach either $\be_0$ or the stable limit cycle (panel (d)); to add clarity we also provide a magnification of the plot in panel (d) near this limit cycle (panel (e)). In case V, $\gamma = \gamma^{(\HM)}$, a homoclinic orbit is present around the stable limit cycle as a separatrix: orbits corresponding to initial conditions lying inside (outside, respectively) approach the limit cycle ($\be_0$, respectively) (panels (f) and (g)). In case VI, $\gamma^{(\HM)} < \gamma < \gamma^{(\FLC)}$, the orbital behaviour is the same, except that the separatrix is now an unstable limit cycle (panels (h) and (i)). In case VII, $\gamma= \gamma^{(\FLC)}$, the two limit cycles coalesce and become a single semistable limit cycle which is approached by orbits corresponding to initial conditions inside it, while orbits corresponding to initial conditions outside it approaches $\be_0$ (panels (j) and (k)). In case VIII, $\gamma^{(\FLC)} < \gamma < \gamma^{(\SN)}$, no limit cycle exists; orbits approach $\be_0$ (panel (l)). 

In case IX, the two endemic equilibria coalesce and become a single unstable endemic equilibrium, and orbits approach the disease-free equilibrium $\be_0$ (panel (m)). In case X, we have the same orbital behaviour, but with no existing endemic equilibria (panel (n)).

In all cases whereby periodic orbits exist (IV to VII), in the magnified plots (panels (e), (g), (i), (k)) we plot in light blue the projections of the periodic orbits on the $SI$-plane, confirming that these orbits indeed intersect the quadratic curves specified in Theorem \ref{prop:periodicorbit} which are plotted in magenta.

\begin{figure}\centering
\begin{tabular}{ccc}
		\includegraphics[width=0.3\textwidth]{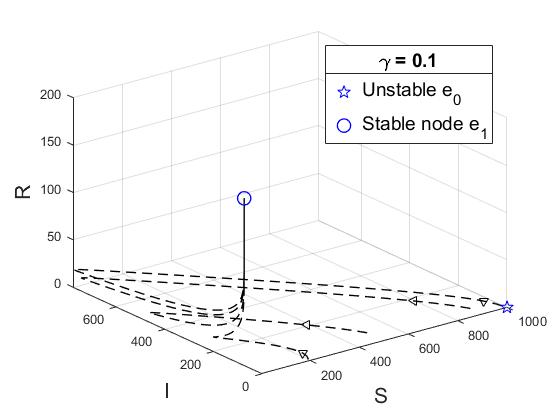}&
		\includegraphics[width=0.3\textwidth]{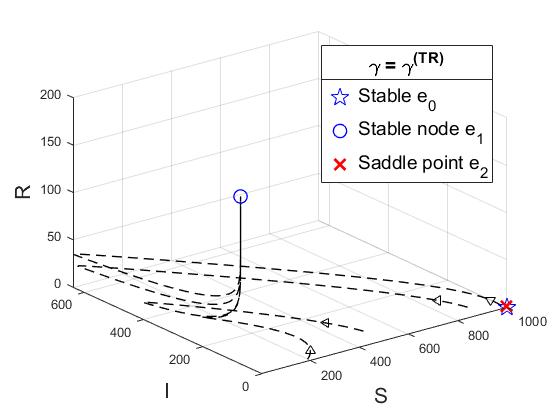}&
		\includegraphics[width=0.3\textwidth]{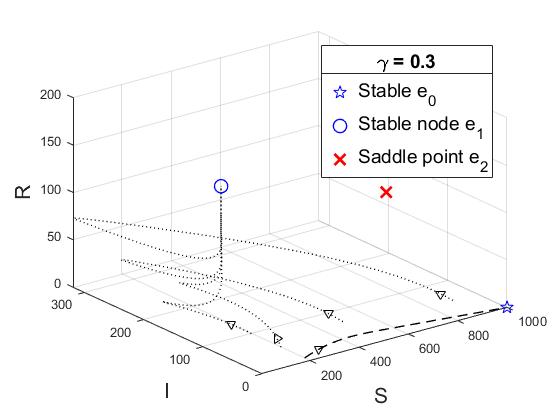}\\
		{\scriptsize (a) Case I: $\gamma = 0.1$} & {\scriptsize (b) Case II: $\gamma = \gamma^{(\TR)}$} & {\scriptsize (c) Case III: $\gamma = 0.3$}\\
		
		\includegraphics[width=0.3\textwidth]{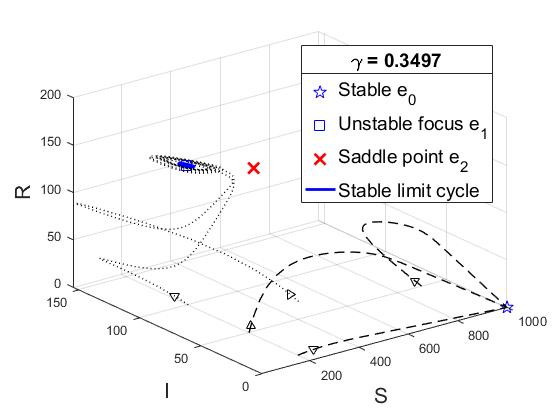}&
		\includegraphics[width=0.3\textwidth]{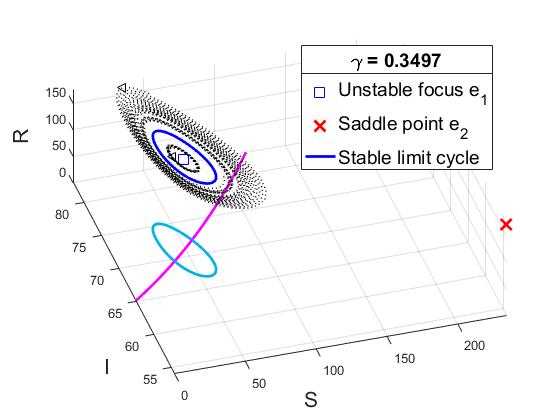}&
		\includegraphics[width=0.3\textwidth]{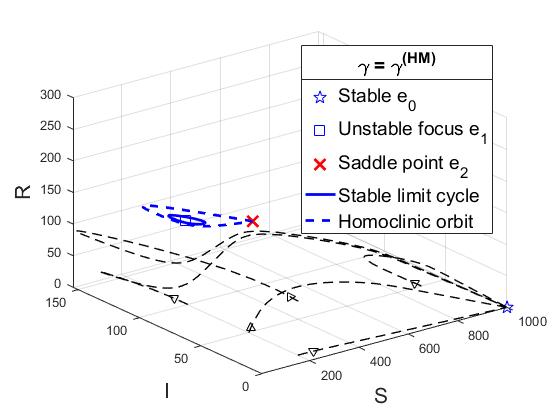}\\
		{\scriptsize (d) Case IV: $\gamma = 0.3497$} & {\scriptsize (e) Case IV: $\gamma = 0.3497$ (magnified)} & {\scriptsize (f) Case V: $\gamma = \gamma^{(\HM)}$}\\

		\includegraphics[width=0.3\textwidth]{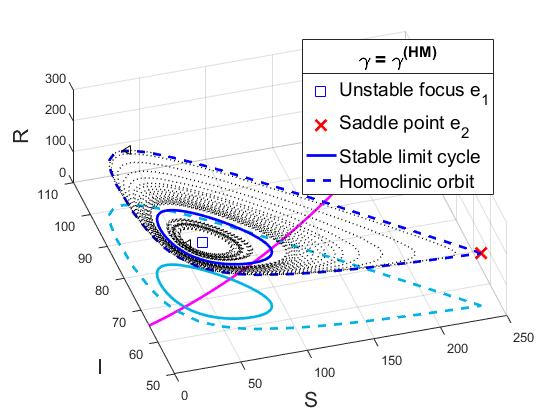}&
		\includegraphics[width=0.3\textwidth]{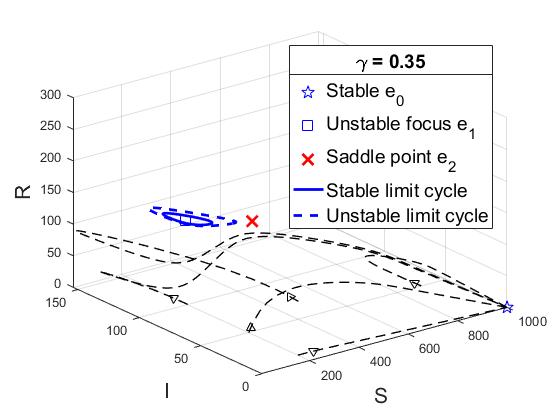}&
		\includegraphics[width=0.3\textwidth]{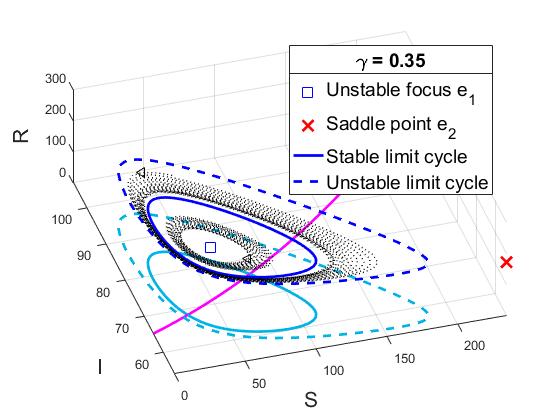}\\	
		{\scriptsize (g) Case V: $\gamma = \gamma^{(\HM)}$ (magnified)} & {\scriptsize (h) Case VI: $\gamma = 0.35$} & {\scriptsize (i) Case VI: $\gamma = 0.35$ (magnified)}\\

		\includegraphics[width=0.3\textwidth]{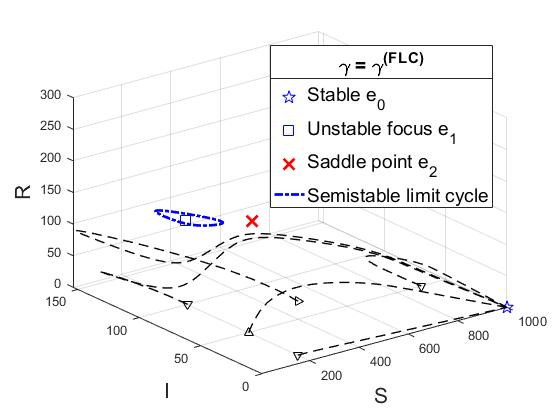}&
		\includegraphics[width=0.3\textwidth]{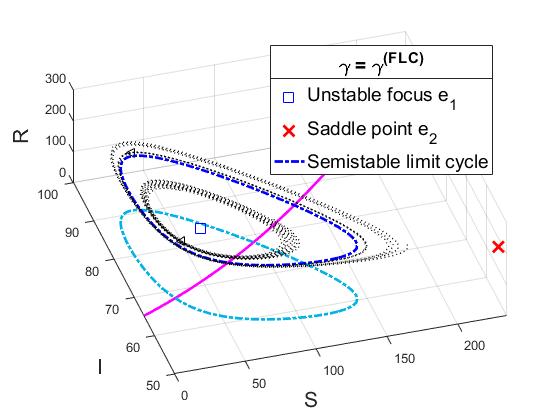}&
		\includegraphics[width=0.3\textwidth]{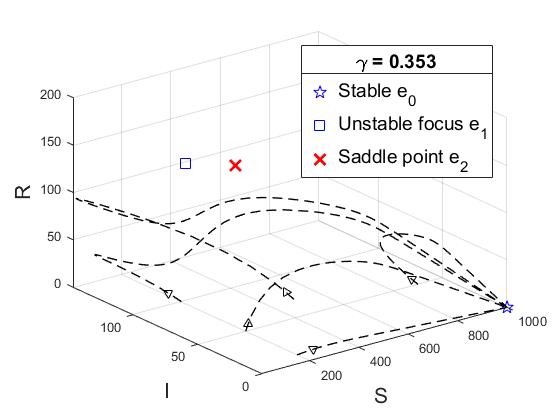}\\	
		{\scriptsize (j) Case VII: $\gamma = \gamma^{(\FLC)}$} & {\scriptsize (k) Case VII: $\gamma = \gamma^{(\FLC)}$ (magnified)} & {\scriptsize (l) Case VIII: $\gamma = 0.353$}\\

		\includegraphics[width=0.3\textwidth]{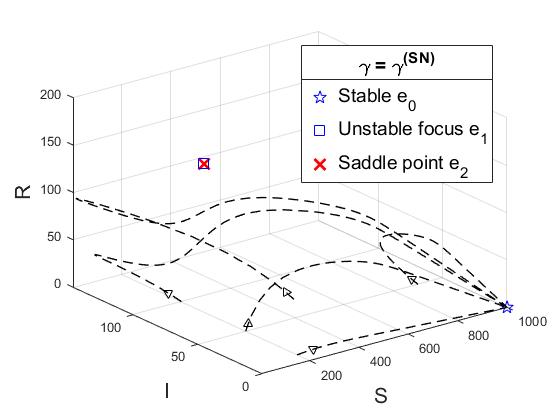}&
		\includegraphics[width=0.3\textwidth]{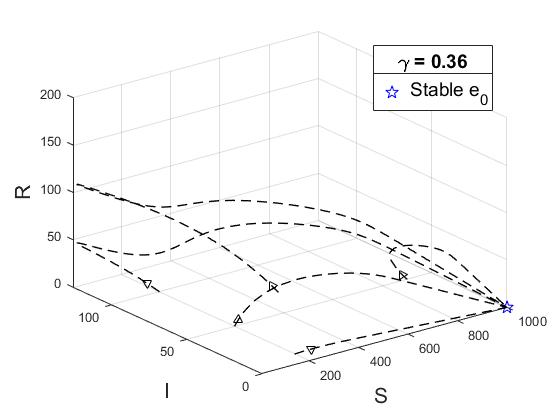}\\	
		{\scriptsize (m) Case IX: $\gamma = \gamma^{(\SN)}$} & {\scriptsize (n) Case X: $\gamma = 0.36$}\\

\end{tabular}		
		\caption{\label{fig:simulations}Phase portraits of the model \eqref{eq:model} in various cases considered in Table \ref{tab:summary}.}
\end{figure}

\subsection{Varying cautiousness level}\label{subsec:numerics3}

To conclude our numerical analysis, let us simulate a scenario whereby, in the equilibrium state, the number $I(t)$ of infected individuals changes over time $t$ as a result of the susceptible individuals' cautiousness level $\gamma(t)$ changing over time $t$. To this end, we first need to choose a specific function $\gamma$ of $t$. We assume that this function is piecewise-constant for simplicity, and construct its branches to accomplish three aims:
\begin{enumerate}
\item[(i)] to expose the hysteresis exhibited by our model,
\item[(ii)] to include in the scenario most of the qualitatively different orbital behaviours of our model as displayed in Figure \ref{fig:simulations}, and
\item[(iii)] to represent, to a certain degree, some actual key events occurring between December 2019 and July 2021 that have affected the spread of COVID-19 in Indonesia.
\end{enumerate}
The construction is described in Table \ref{tab:events}. A plot of $\gamma(t)$ versus $t$ is shown in Figure \ref{fig:scenario} (top).

\begin{table}\centering
\scalebox{0.85}{\begin{tabular}{|l|l|l|l|l|}
\hline
Time interval & Range of $t$ & Events & Value of $\gamma(t)$ & Case\\
\hhline{|=|=|=|=|=|}
Dec 2019 -- Jan 2020 & $0\leqslant t<200$ & \begin{minipage}{5cm}\smallskip Panic began as the world's first COVID-19 case was recorded in China \cite{Saxena,WHOI26Mar2020}. \end{minipage}& 0.3 & III\\\hline
Jan 2020 -- Mar 2020 & $200\leqslant t<600$ & \begin{minipage}{5cm}\smallskip Government urged citizens not to panic \cite{Desk} and provided incentives for foreign visitors, boosting tourism \cite{Gorbiano}.\end{minipage} & 0.1 & I\\\hline
Mar 2020 -- Apr 2020 & $600\leqslant t<800$ & \begin{minipage}{5cm}\smallskip Indonesia's first COVID-19 case was recorded \cite{WHOI26Mar2020}, the large-scale social restrictions (PSBB) was announced \cite{WHOI09Apr2020}, and the disease spread across all 34 provinces \cite{WHOI16Apr2020}.\end{minipage} & 0.33 & \multirow{9}{*}{\begin{minipage}{0.7cm}\bigskip\bigskip\bigskip\bigskip\bigskip\bigskip\bigskip\bigskip\bigskip\bigskip\bigskip\bigskip\bigskip\bigskip\bigskip\bigskip\bigskip\bigskip III\end{minipage}}\\\cline{1-4}
Apr 2020 -- Jun 2020 & $800\leqslant t<1200$ & \begin{minipage}{5cm}\smallskip Government banned \textit{mudik} (the Eid al-Fitr homecoming of migrant workers), albeit not fully obeyed \cite{WHOI16Apr2020}.\end{minipage} & 0.32 &\\\cline{1-4}
Jun 2020 -- Sep 2020 & $1200\leqslant t<1800$ & \begin{minipage}{5cm}\smallskip Government of Jakarta announced a gradual transition from PSBB to `new normal', restimulating citizens' mobility \cite{WHOI10Jun2020}.\end{minipage} & 0.31 &\\\cline{1-4}
Sep 2020 -- Oct 2020 & $1800\leqslant t<2000$ & \begin{minipage}{5cm}\smallskip Government of Jakarta reimposed PSBB \cite{WHOI16Sep2020}.\end{minipage} & 0.315 &\\\cline{1-4}
Oct 2020 -- Dec 2020 & $2000\leqslant t<2400$ & \begin{minipage}{5cm}\smallskip Government of Jakarta relaxed PSBB, introducing `transitional PSBB' \cite{WHOI14Oct2020}. \end{minipage} & 0.305 &\\\cline{1-4}
Dec 2020 -- Feb 2021 & $2400\leqslant t<2800$ & \begin{minipage}{5cm}\smallskip A new highest weekly incidence was recorded following Christmas holidays \cite{WHOI20Jan2021}. \end{minipage} & 0.32 &\\\cline{1-4}
Feb 2021 -- Apr 2021 & $2800\leqslant t<3200$ & \begin{minipage}{5cm}\smallskip Government announced micro-level restrictions on community activities (micro-level PPKM) \cite{WHOI10Feb2021}. \end{minipage} & 0.31 &\\\cline{1-4}
Apr 2021 -- May 2021 & $3200\leqslant t<3400$ & \begin{minipage}{5cm}\smallskip Citizens ignored the government's second \textit{mudik} ban \cite{SyakriahSuhendraThomas}. \end{minipage} & 0.305 &\\\cline{1-4}
May 2021 -- Jun 2021 & $3400\leqslant t<3600$ & \begin{minipage}{5cm}\smallskip A surge in the number of COVID-19 cases and clusters of COVID-19 infection were recorded following Eid al-Fitr holidays \cite{WHOI02Jun2021}. \end{minipage} & 0.34 & \\\hline
Jun 2021 -- Aug 2021 & $3600\leqslant t<4000$ & \begin{minipage}{5cm}\smallskip Government announced emergency restrictions on community activities (emergency PPKM) \cite{WHOI07Jul2021}. \end{minipage} & 0.3497 & IV\\\hline
Aug 2021 -- Sep 2021 & $4000\leqslant t<4200$ & \begin{minipage}{5cm}\smallskip Government rewarned the public to strictly adhere to health protocols, as the Delta variant of the virus was reported in 24 of 34 provinces \cite{WHOI04Aug2021}. \end{minipage} & 0.35 & VI\\\hline
\end{tabular}}	
\caption{\label{tab:events}Constructing a piecewise-constant function $\gamma$ of $t$ for a simulation, in view of some events affecting the spread of COVID-19 in Indonesia between December 2019 and July 2021.}
\end{table}

\begin{figure}\centering
\includegraphics[width=\textwidth]{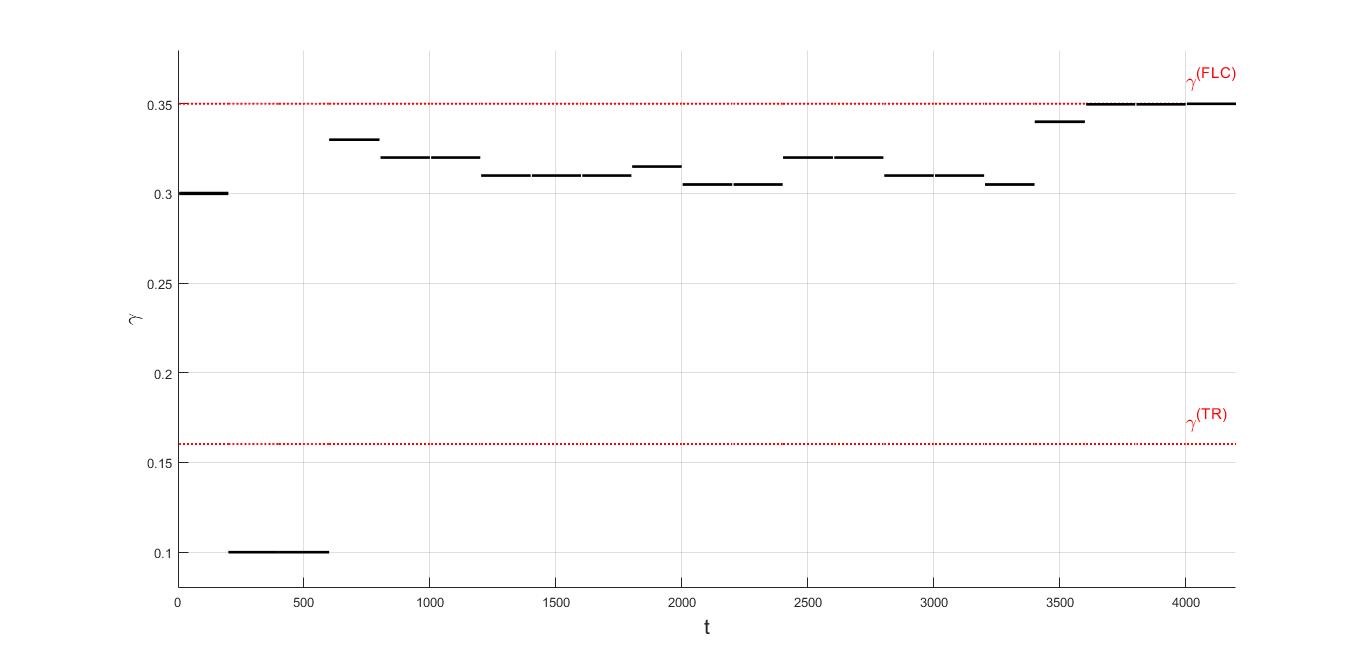}\\
\includegraphics[width=\textwidth]{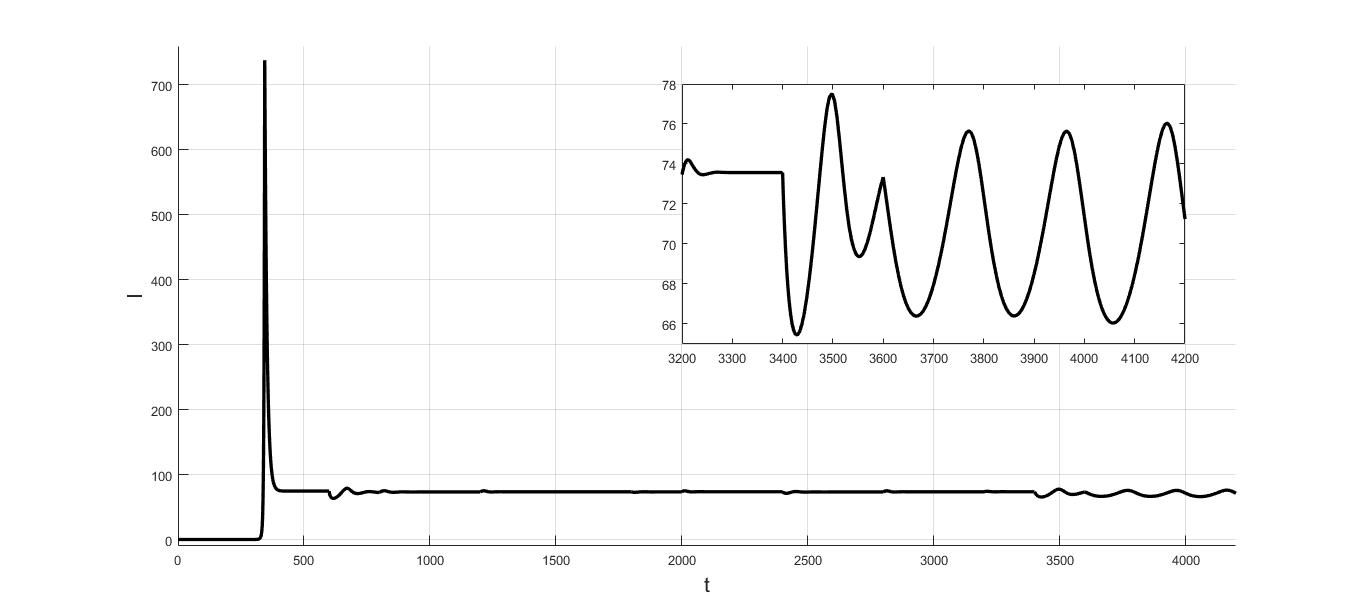}
\caption{\label{fig:scenario} Plots of $\gamma(t)$ versus $t$ (top), and $I(t)$ versus $t$ with a magnification in the region $[3200,4200]\times[65,78]$ (bottom).}
\end{figure}

Using the initial condition $\left(S_0,I_0,R_0\right)=(100,0.001,0)$, we obtain the plot of $I(t)$ versus $t$ displayed in Figure \ref{fig:scenario} (bottom). Let us now describe and interpret this plot biologically, and see that our model undergoes a \textit{hysteresis} along the cycle plotted in Figure \ref{fig:hysteresis}.

\subsubsection*{The disease-free beginning} First, we consider the period $0\leqslant t<200$. We begin with $I(0)=0.001$, essentially meaning that no individual in the system is initially infected, and assume that the susceptible individuals have a moderate cautiousness level: $\gamma(t)=0.3$ (case III in Table \ref{tab:summary}). Consequently, the disease-free equilibrium is stable, attracting the system (Figure \ref{fig:simulations} (c)), making it remain disease-free throughout the period. In Figure \ref{fig:hysteresis}, the system's current status is represented by point 1. 

\subsubsection*{The start of the pandemic} For $200\leqslant t<600$, we assume that individuals have become less cautious: $\gamma(t)=0.1<\gamma^{(\TR)}$ (case I in Table \ref{tab:summary}). For such a low cautiousness level, the disease-free equilibrium is unstable. The disease thus enters the system, quickly driving the system to the stable endemic equilibrium (Figure \ref{fig:simulations} (a)). In Figure \ref{fig:hysteresis}, the system is now at point 2.

\subsubsection*{The early effort for recovery} Next, we consider the period $600\leqslant t< 3600$. We assume that, for $600\leqslant t<800$, the cautiousness level is restored to its value before the start of the pandemic (in fact, slightly exceeding it): $\gamma(t)=0.33$ (case III in Table \ref{tab:summary}). This restoration, although results in a slight decrease in $I(t)$, does \textit{not} suffice to bring the system back to the disease-free equilibrium. Indeed, for this value of $\gamma(t)$ our system is \textit{bistable}: not only the disease-free equilibrium, but also the endemic equilibrium the system is currently in the vicinity of, are both present and stable (Figure \ref{fig:simulations} (c)). In Figure \ref{fig:hysteresis}, the system's status is now represented by point 3. The value of $\gamma(t)$ during the remaining period $800\leqslant t< 3600$, albeit changing, is within the range of case III in Table \ref{tab:summary}; thus the changes do not result in a qualitatively different orbital behaviour. In order to bring the system back to the disease-free equilibrium, effort must be made to further increase the cautiousness level: $\gamma(t)$ must exceed $\gamma^{(\FLC)}$ to achieve a situation whereby the disease-free equilibrium is the only attractor (cases VIII to X in Table \ref{tab:summary}).

\subsubsection*{The subsequent effort} In the period $3600\leqslant t<4200$, we assume an increased cautiousness level: $\gamma(t)=0.3497\in\left(\gamma^{(\HB)},\gamma^{(\HM)}\right)$ for $3600\leqslant t<4000$ (case IV in Table \ref{tab:summary}) and $\gamma(t)=0.35\in\left(\gamma^{(\HM)},\gamma^{(\FLC)}\right)$ for $4000\leqslant t<4200$ (case VI in Table \ref{tab:summary}). For these values of $\gamma(t)$, the system, being previously in the vicinity of the stable endemic equilibrium, now circulates around a stable periodic orbit (Figure \ref{fig:simulations} (d), (e), (h), (i)). The decreasing parts of the resulting oscillations, in reality, may suggest that the pandemic is settling down so that the system will soon become disease-free; this is certainly not the case.

\subsubsection*{What could happen next?} Either the value of $\gamma(t)$ remains below $\gamma^{(\FLC)}$ forever, or at some point it exceeds $\gamma^{(\FLC)}$. The former means that the system remains endemic forever, while the latter means that the system successfully returns to the disease-free state (cases VIII to X in Table \ref{tab:summary}; Figure \ref{fig:simulations} (l), (m), (n); point 4 in Figure \ref{fig:hysteresis}). Once this has taken place, individuals need not retain their high cautiousness level; they can safely reduce their cautiousness level, as long as it remains above $\gamma^{(\TR)}$, without bringing the system back to the endemic state. In practice, this could mean that, once the pandemic has settled down, individuals no longer need to observe strict health protocols. Indeed, the cautiousness level $\gamma(t)$ can safely be set as low as $0.17$. If it becomes $0.16$, however, the pandemic reattacks, and, as before, necessitates a significant effort to settle it down: $\gamma(t)$ must once again be brought to exceed $\gamma^{(\FLC)}$. This circulatory dynamical behaviour is best visualised as the hysteresis cycle in Figure \ref{fig:hysteresis}.

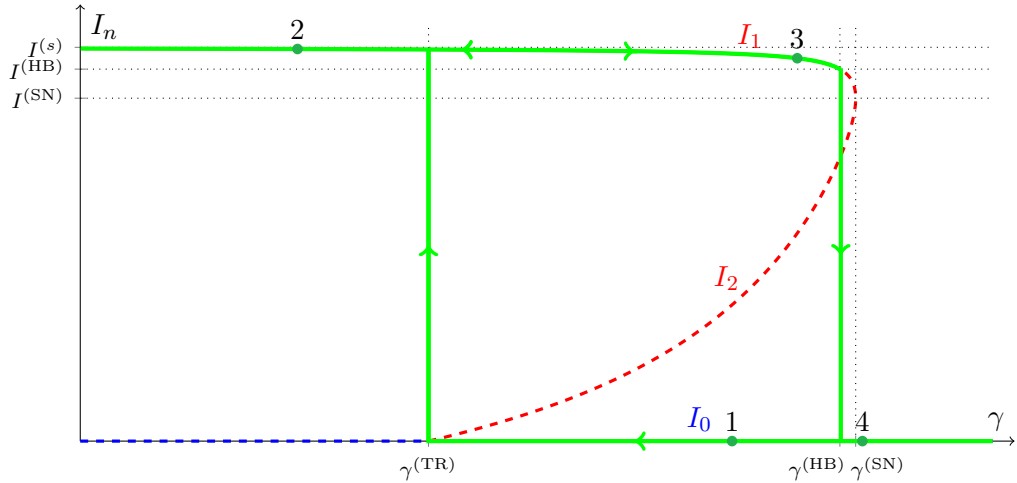
\begin{figure}\centering
\begin{tikzpicture}
\begin{axis}[
	xmin=0,
	xmax=0.43,
	ymin=0,
	ymax=83,
    xtick={0.1602903226,.3496365152,.3569024928},
    xticklabels={$\gamma^{(\TR)}$,$\gamma^{(\HB)}\,\,\,\,\,\,\,\,\,\,\,\,$,$\,\,\,\,\,\,\,\,\,\,\,\,\gamma^{(\SN)}$},    
    ytick={65.19550517,70.72118905,74.86959682},
    yticklabels={$I^{(\SN)}$,$I^{(\HB)}$,$I^{(s)}$},
	samples=100,
	xlabel=$\gamma$,
	ylabel=$I_n$,
	width=0.9\textwidth,
    height=0.475\textwidth,
    axis lines=middle,
    x axis line style=->,
    y axis line style=->,
    clip=false,
    tick label style={font=\scriptsize},
]
\addplot [domain = 0:70.72118905, samples = 300, very thick, red, dashed]
      ({(.4545454545*(x-74.63262880))*(x+10.00000002)*(x+12.10535604)/((x-74.86959682)*(x+12.14232409)*(x+28.18181818))}, {x});
\addplot [domain = 70.72118905:74.63262879, samples = 300, very thick, red]
      ({(.4545454545*(x-74.63262880))*(x+10.00000002)*(x+12.10535604)/((x-74.86959682)*(x+12.14232409)*(x+28.18181818))}, {x});
\addplot [domain = 0:0.1602903226, samples = 300, very thick, blue, dashed]
      ({x}, {0});
\addplot [domain = 0.1602903226:0.42, samples = 300, very thick, blue]
      ({x}, {0});
\draw[dotted] (axis cs:.1602903226,0) -- (axis cs:.1602903226,79);
\draw[dotted] (axis cs:.3496365152,0) -- (axis cs:.3496365152,79);
\draw[dotted] (axis cs:.3569024928,0) -- (axis cs:.3569024928,79);
\node[blue,above] at (axis cs:0.285,0) {$I_0$};
\node[red] at (axis cs:0.2975,31) {$I_2$};
\node[red] at (axis cs:0.3082443924,77) {$I_1$};
\draw[dotted] (axis cs:0,74.86959682)--(axis cs:0.42,74.86959682);
\draw[dotted] (axis cs:0,70.72118905)--(axis cs:0.42,70.72118905);
\draw[dotted] (axis cs:0,65.19550517)--(axis cs:0.42,65.19550517);

\draw[postaction={decorate},ultra thick,green] (axis cs:0.42,0)--(axis cs:0.1602903226,0)--(axis cs:0.1602903226,74.45478081);

\begin{scope}[ultra thick,decoration={
    markings,
    mark=at position 0.5 with {\arrow{>}}}
    ] 
    \draw[postaction={decorate},green] (axis cs:0.1602903226,0)--(axis cs:0.1602903226,74.45478081);
\addplot [postaction={decorate},domain = 70.62034708:74.63262879, samples = 300, green]
      ({(.4545454545*(x-74.63262880))*(x+10.00000002)*(x+12.10535604)/((x-74.86959682)*(x+12.14232409)*(x+28.18181818))}, {x});
    \draw[postaction={decorate},green] (axis cs:0.35005851840,70.8)--(axis cs:0.35005851840,0);
    \draw[postaction={decorate},green] (axis cs:0.2561744205,0)--(axis cs:0.2541744205,0);
    \draw[postaction={decorate},green] (axis cs:0.2541744205,74.12831636)--(axis cs:0.2561744205,74.11564743);
\end{scope}

\fill[darkgreen] (axis cs:0.3,0) node[above,black] {$1$} circle (2pt);
\fill[darkgreen] (axis cs:0.1,74.54612389) node[above,black] {$2$} circle (2pt);
\fill[darkgreen] (axis cs:0.33,72.79426361) node[above,black] {$3$} circle (2pt);
\fill[darkgreen] (axis cs:0.36,0) node[above,black] {$4$} circle (2pt);

\end{axis}
\end{tikzpicture}
\caption{\label{fig:hysteresis} The hysteresis cycle exhibited by our model \eqref{eq:model}.}
\end{figure}

\section{Conclusions and future research}\label{sec:conclusions}

We have studied an SIR-type model for the spread of COVID-19 in a population. The model incorporates as parameters the population's entrance and exit rates, incidence and recovery coefficients, the susceptible individuals' cautiousness level, and the hospitals' bed-occupancy rate. From its analysis, we draw three conclusions.

First, we have proved that the models' basic reproduction number does not depend on the bed-occupancy rate, suggesting that:
\begin{enumerate}
\item[(i)] merely increasing the hospitals' bed-occupancy rate (e.g., by establishing increasingly many makeshift hospitals) does not resolve the pandemic.
\end{enumerate}
Next, since suppressing the entrance rate or the incident coefficient could lead to economic damage, increasing the exit rates is undesirable, and increasing the recovery coefficient is difficult due to the limitedness of medical resources, the model has suggested that:
\begin{enumerate}
\item[(ii)] the best way to resolve the pandemic is to increase the susceptible individuals' cautiousness level.
\end{enumerate}
Most importantly, we have observed the importance of a high cautiousness level for resolving the pandemic. Indeed, the system, having been brought to an endemic state even by a slight drop of cautiousness level, cannot be recovered to disease-free by merely restoring the cautiousness level to its pre-endemic value. We conclude that:
\begin{enumerate}
\item[(iii)] for a successful annihilation of the pandemic, a high cautiousness level is crucial.
\end{enumerate}
More plainly, whether or not our wish that the pandemic soon ends will become a reality relies very largely on whether or not we are seriously cautious of the disease and consequently adhere to health protocols. It is \textit{not} sufficient for us to hold the cautiousness level that we had before the pandemic entered our country: perhaps merely being aware of the existence of the disease without implementing protective actions.

This research is extendible in a number of ways. In the numerical analysis, instead of varying only the cautiousness level, one could also vary the bed-occupancy rate, thereby investigating codimension-two bifurcations exhibited by the model. For increased values of the bed-occupancy rate, preliminary experiments reveal the occurrence of Bogdanov-Takens and generalised Hopf bifurcations, as well as the existence of a region whereby a stable endemic equilibrium is surrounded by an unstable limit cycle. 

Moreover, the model studied in this paper is simplified; it can be made more realistic by adding more compartments, e.g., those of quarantined, exposed, and/or asymptomatic individuals. In addition, one could take into account, e.g., the incubation period, and/or the possibility of reinfection, introducing a positive rate at which recovered individuals become susceptible.




\end{document}